\documentclass[journal,twoside,web]{ieeecolor}

\usepackage{pifont}

\usepackage[T1]{fontenc}
\usepackage{generic}
\usepackage{amsmath,amssymb,bbm}
\let\labelindent\relax
\usepackage{enumitem}
\usepackage{cancel}
\usepackage{graphicx}
\usepackage{xcolor}
\usepackage{diagbox}
\DeclareMathAlphabet{\mathbbold}{U}{bbold}{m}{n}

\usepackage{makecell}

\usepackage{algorithm}
\usepackage{algpseudocode}
\algrenewcommand\algorithmicrequire{\textbf{Input:}}
\algrenewcommand\algorithmicensure{\textbf{Output:}}

\renewcommand\algorithmicdo{}

\newcommand\NoDo{\renewcommand\algorithmicdo{}}

\usepackage{accents}

\definecolor{mygray}{rgb}{0.7, 0.7, 0.7}

\newcommand{\rea}{\mathbb{R}}

\newcommand{\reap}{\mathbb{R}_+}
\newcommand{\nat}{\mathbb{N}}

\renewcommand{\int}{\mathbb{I}}

\newcommand{\dsh}{d_{\textsc{sh}}}

\newcommand{\bone}{\mathbbold{1}}
\newcommand{\bzero}{\mathbbold{0}}
\renewcommand{\int}{\mathbb{Z}}

\newcommand{\A}{\mathcal{A}}
\newcommand{\C}{\mathcal{C}}
\newcommand{\D}{\mathcal{D}}
\newcommand{\R}{\mathcal{R}}
\newcommand{\I}{\mathcal{I}}
\newcommand{\J}{\mathcal{J}}
\newcommand{\Y}{\mathcal{Y}}

\newcommand{\G}{\mathcal{G}}
\renewcommand{\P}{\mathcal{P}}

\newcommand{\V}{\mathcal{V}}
\newcommand{\E}{\mathcal{E}}
\newcommand{\N}{\mathcal{N}}
\newcommand{\X}{\mathcal{X}}
\renewcommand{\S}{\mathcal{S}}

\newcommand{\Va}{\V^{\textsc{A}}}
\newcommand{\Vr}{\V^{\textsc{R}}}
\newcommand{\Vd}{\V^{\textsc{D}}}

\newcommand{\x}{\boldsymbol{x}}
\newcommand{\y}{\boldsymbol{y}}

\newcommand{\Ts}{\mathsf{T}}

\newcommand{\Ps}{\mathsf{P}}
\newcommand{\Fs}{\mathsf{F}}

\DeclareMathOperator{\prox}{prox}
\DeclareMathOperator{\proj}{proj}

\DeclareMathOperator{\avg}{avg}
\DeclareMathOperator{\med}{med}

\DeclareMathOperator*{\arginf}{arginf}

\DeclareMathOperator*{\blkdiag}{blkdiag}

\renewcommand{\bar}[1]{\overline{#1}}
\newcommand{\ubar}[1]{\underline{#1}}

\newcommand{\alg}{Open ADMM\xspace}

\newcommand{\join}{^\text{join}}
\newcommand{\leave}{^\text{leave}}

\newcommand{\norm}[1]{{\left\vert\kern-0.25ex\left\vert #1\right\vert\kern-0.25ex\right\vert}}
\newcommand{\mnorm}[1]{{\left\vert\kern-0.3ex\left\vert\kern-0.3ex\left\vert#1\right\vert\kern-0.3ex\right\vert\kern-0.3ex\right\vert}}
\newcommand{\abs}[1]{{\left\vert #1 \right\vert}}
\newcommand{\norminf}[1]{{\left\vert\kern-0.25ex\left\vert #1\right\vert\kern-0.25ex\right\vert}_{ \infty}}

\usepackage{xspace}

\newcommand{\virg}[1]{``#1''}

\usepackage{ntheorem,lipsum}
\theoremheaderfont{\hspace*{\parindent}\bf}
\theoremseparator{.}
\newtheorem{assum}{Assumption}
\newtheorem{thm}{Theorem}
\newtheorem{rem}{Remark}
\newtheorem{prop}{Proposition}
\newtheorem{cor}{Corollary}
\newtheorem{lem}{Lemma}
\newtheorem{defn}{Definition}
\newtheorem{exmp}{Example}

\newcommand{\cmark}{{\color{green!80!black}\ding{51}\xspace}}
\newcommand{\xmark}{{\color{red}\ding{55}\xspace}}
\newcommand{\mmark}{{\color{orange}\ding{91}\xspace}}

\makeatletter
\let\c@author\relax
\makeatother
\usepackage[citestyle=numeric-comp,
	 bibstyle=ieee,
	 giveninits=true,
      maxbibnames=99,
	 ]{biblatex}

\AtEveryBibitem{
	\clearfield{issn}
	\clearfield{isbn}
	\clearfield{archivePrefix}
	\clearfield{arxivId}
	\clearfield{pmid}
	\clearfield{doi}
	\clearfield{eprint}
	\clearfield{note}
	\clearfield{month}
	\ifentrytype{online}{}{\ifentrytype{misc}{}{\clearfield{url}}}
	\ifentrytype{book}{\clearfield{doi}}{}
}

\usepackage[shortcuts]{extdash}
\begin{document}

\title{\fontsize{22pt}{28pt}\selectfont Optimization and Learning \\in Open Multi-Agent Systems}
\author{Diego Deplano, \IEEEmembership{Member, IEEE}, Nicola Bastianello, \IEEEmembership{Member, IEEE}\\
Mauro Franceschelli, \IEEEmembership{Senior, IEEE}, and Karl H. Johansson, \IEEEmembership{Fellow, IEEE}
\thanks{The work of D. Deplano was supported by the project e.INS- Ecosystem of Innovation for Next Generation Sardinia (cod. ECS 00000038) funded by the Italian Ministry for Research and Education (MUR) under the National Recovery and Resilience Plan (NRRP) - MISSION 4 COMPONENT 2, \virg{From research to business} INVESTMENT 1.5, \virg{Creation and strengthening of Ecosystems of innovation} and construction of \virg{Territorial R\& D Leaders}.}
\thanks{The work of N. Bastianello and K. H. Johansson was partially supported by the European Union’s Horizon Research and Innovation Actions programme under grant agreement No. 101070162, and partially by Swedish Research Council Distinguished Professor Grant 2017-01078 Knut and Alice Wallenberg Foundation Wallenberg Scholar Grant.}
\thanks{D. Deplano and M. Franceschelli is with DIEE, University of Cagliari, 09123 Cagliari, Italy.
Emails: {\tt \{diego.deplano,mauro.franceschelli\}@unica.it}}
\thanks{Nicola Bastianello and Karl H. Johansson are with the School of Electrical Engineering and Computer Science and Digital Futures, KTH Royal Institute of Technology, Stockholm, Sweden.
Emails: {\tt \{nicolba,kallej\}@kth.se}}
}

\maketitle
\begin{abstract}
Modern artificial intelligence relies on networks of agents that collect data, process information, and exchange it with neighbors to collaboratively solve optimization and learning problems.
This article introduces a novel distributed algorithm to address a broad class of these problems in \virg{open networks}, where the number of participating agents may vary due to several factors, such as autonomous decisions, heterogeneous resource availability, or DoS attacks.
Extending the current literature, the convergence analysis of the proposed algorithm is based on the newly developed \virg{Theory of Open Operators}, which characterizes an operator as open when the set of components to be updated changes over time, yielding to time-varying operators acting on sequences of points of different dimensions and compositions.
The mathematical tools and convergence results developed here provide a general framework for evaluating distributed algorithms in open networks, allowing to characterize their performance in terms of the punctual distance from the optimal solution, in contrast with regret-based metrics that assess cumulative performance over a finite-time horizon.
As illustrative examples, the proposed algorithm is used to solve dynamic consensus or tracking problems on different metrics of interest, such as average, median, and min/max value, as well as classification problems with logistic loss functions.
\end{abstract}

\vspace{-1em}


\vspace{-1em}

\section{Introduction}
Many real-world systems consists of multiple interacting agents in a network, where new agents may join (start interacting) and others may leave (stop interacting), forming what is called an \virg{\textit{open}} multi-agent systems:
Bitcoin miners compete in a network to gain a reward by adding a block of transactions to the blockchain, leaving after winning or when the profitability does not justify the costs \cite{nakamoto2008bitcoin};
mobile robots cooperate in a network to achieve global tasks such as move goods or perform inventory checks, leaving the network to charge their battery or due to malfunctioning \cite{brambilla2013swarm};
people participate in a network to share their thoughts and learn from others, leaving when they lose interest \cite{rainer2002opinion}.

\begin{table*}[!ht]
\centering

\caption{Comparison with the state of the art for distributed optimization and learning\\ in open multi-agent systems with discrete-time dynamics.}
\label{tab:comparison_optimization}
    \begin{tabular}{lcccccc}
    \hline
    [Ref.] & \thead{Algorithm} & \thead{Assumptions\\on the problem} & \thead{Assumptions\\on the network} & \thead{Time independent\\step sizes} & \thead{Convergence \\ metric} & \thead{Convergence\\ rate} \\
    \hline
    \thead[l]{\cite{hendrickx_stability_2020} Hendrickx \textit{et al.} \\\hspace{1.5em}(2020)} & DGD & \thead{Static + Smooth + \\ Strongly convex + \\Minimizers in a ball} & \thead{\xmark\\ Only replacement} & \cmark & \thead{Distance from \\ minimizers}& \thead{Linear\\ (inexact)} \\

    \thead[l]{\cite{hsieh_optimization_2021} Hsieh \textit{et al.} \\\hspace{1.5em}(2021)}& \thead{Dual\\averaging} & \thead{Static +\\ Lipschitz + Convex +\\Shared convex constraint set} & \thead{\mmark\\ Vertex-connected$^\dagger$\\(jointly)} & \xmark & Regret &  \thead{Sublinear\\(if the network's\\size is known)} \\

    \thead[l]{\cite{hayashi_distributed_2023} Hayashi \\\hspace{1.5em}(2023)} & Sub-gradient & \thead{Time-varying +\\ Lipschitz + Convex \\Shared compact constraint set} & \thead{\mmark\\ Vertex-connected$^\dagger$\\(jointly)} & \xmark & Regret & \thead{Sublinear\\(if the network's \\size is bounded)} \\

    \hline

    [\textbf{This work}] & ADMM & \thead{Time-varying + \\Semicontinuous + Convex\\Unconstrained} & \thead{\cmark\\ Connected} & \cmark & \thead{Distance from \\ minimizers} & \thead{Linear \\ (exact)}\\
    
    \hline
    \end{tabular}
\end{table*}

Despite observing this dynamic behavior in practice, traditional cooperation schemes, optimization algorithms, game strategies, and learning techniques often assume a static network with a fixed set of participants.
Under this simplified assumption, it is typically possible to prove the stability of these methods and characterize their performance in reaching equilibrium points, representing best-response strategies in games, optimal solutions in optimization tasks, or finest trained models in machine learning applications.
Such guarantees may lose their relevance in the context of open networks, because an equilibrium may never be reached and divergent behavior could arise due to the join/leave events which the network is subject.
Within the control community, a growing interest in the study of open networks is demonstrated by the high number of papers published within the last decade, where the major topics of interest are represented by distributed consensus protocols, both in continuous-time \cite{xue2022stability, zhou2024prescribed, restrepo2022consensus, dashti2019dynamic} and in discrete-time \cite{varma2018open, de2019lower, de2022fundamental, dashti2022distributed,hendrickx2016opendet,hendrickx2017openrand, franceschelli2018proportional, franceschelli2020stability, abdelrahim2017max, Deplano24stability, oliva2023sum, de2021random}, distributed optimization algorithms \cite{hsieh_optimization_2021, hsieh2022multi, hayashi_distributed_2023, sawamura_distributed_2024, hendrickx_stability_2020}, distributed resource allocation \cite{liu_distributed_2024, vizuete2022resource, galland_random_2024}, and learning problems \cite{bistritz2024gamekeeper, nakamura2023cooperative}.
The difficulty in designing distributed algorithms that can be deployed over open networks, and in providing formal performance guarantees, is mainly due to the lack of formal mathematical tools to analyze the dynamics of systems with a varying number of components.

To fill this gap, the \textbf{first contribution} of this manuscript is the formalization of fundamental concepts for what we call \virg{\textit{open operators}}, i.e., time-varying operators acting on sequences of points of different dimensions and compositions.
We extend usual concepts for standard operators \cite{Ryu16,Bauschke2017,} -- such as distance between points and sets, projections of points into sets, fixed points and convergence of the iteration toward fixed points -- to the set-up of open operators, along with illustrative examples.
Since exact convergence cannot be reached in general due to the varying nature of the set of fixed point, both in size and composition, this manuscript provide sufficient conditions for a more general class of contractive operators, ensuring the convergence of the iteration to a bounded distance from the fixed point set.

The \textbf{second contribution} of this manuscript is the presentation and characterization an open and distributed version of ADMM to solve optimization and learning problems in a distributed way over open networks. \let\thefootnote\relax\footnote{$^\dagger$A graph $\G_k=(\V_k,\E_k)$ is said to be jointly vertex connected if there exists $B,\kappa\geq 1$ such that at least $\kappa$ nodes need to be removed to disrupt the connectivity of the union graph $\big( \bigcup_{t = k-B}^k \mathcal{V}_t, \bigcup_{t = k-B}^k \mathcal{E}_t \big)$ for all $k\in\nat$.}
We call the proposed algorithm \virg{Open ADMM} which enjoys the following advantages if compared with other state-of-the-art algorithms (cfr.~Table \ref{tab:comparison_optimization}): 1) it accommodates arbitrary changes in the network, including unbounded growth, and requires no conditions beyond simple connectivity; 2) it never needs a centralized re-initialization procedure as the step sizes are neither vanishing or time-varying; 3) it works with time-varying cost functions that are only convex and not necessarily strongly convex; 4) it converges linearly to the set of minimizers, achieving exact convergence if the network composition and the local costs remains unchanged for sufficiently long time.
{Assuming all cost functions are static, smooth, strongly convex, and that their minimizers lie within a given ball, the authors of \cite{hendrickx_stability_2020} have shown that the Decentralized Gradient Descent (DGD) Algorithm remains stable in the constrained scenario where departing agents are immediately replaced by new arrivals.
Relaxing strong convexity to convexity and smoothness to Lipschitz continuity, the authors of \cite{hsieh_optimization_2021} proposed a dual averaging method and proved sublinear convergence of the running loss, provided the running ratio of the quadratic mean to the average number of active agents is bounded, which accommodates unbounded network growth.
Under the stronger assumption that the network's size is bounded from above, the author of \cite{hayashi_distributed_2023} introduced a subgradient approach and proved sublinear convergence of the running loss while accounting with time-varying local loss functions.
Unlike this manuscript and \cite{hendrickx_stability_2020}, these works analyze performance using a regret-based metric and require the network to be (jointly) vertex-connected$^\dagger$.
Instead, this work only assumes the network to be connected and characterizes the punctual distance of the local estimates from the global optimum, making the performance evaluation less coarse and more informative than regret-based metrics.}

As a \textbf{third contribution}, this manuscript applies the proposed algorithm to solve dynamic consensus problems over open networks. Specifically, we derive closed-form updates for Open ADMM and establish sufficient conditions on the signals being tracked to ensure the correct tracking of three different metrics of interest: the maximum, the median, and the average of a set of signals, each of which is locally accessible to the agents.
Notably, for the median metric, this work introduces the first discrete-time protocol in the literature designed for open scenarios. Additionally, the proposed approach demonstrates superior performance for both the maximum and average metrics. We refer the interested reader to Section \ref{sec:simu_dc} and specifically to Remark \ref{rem:sota_dc} for a comparison of performance and working assumption between our algorithms and those in the state of the art.

As a \textbf{fourth contribution}, this manuscript applies the proposed algorithm to learning problems with logistic loss functions and provide thorough numerical simulations demonstrating its performance under various scenarios: (i) in networks with Poisson arrivals and departures, the distance from the optimal solution is proportional to the Poisson rates but remains bounded; (ii) in eventually closed networks, with decaying Poisson rates, the algorithm achieves exact convergence at steady state; (iii) in networks with only node replacements, larger networks exhibit higher robustness.
Additionally, alternative initialization methods for Open ADMM are explored to address cases where computing local optima is computationally expensive.

\textit{Structure of the manuscript:}
Section \ref{sec:OOT} introduces all relevant concepts at the root of open operator theory, generalizing the standard concepts of fixed points, sequences and their convergence, projections, and distances. 
Section \ref{sec:algorithm} formalize the problem of interest together with the working assumptions, then presents and characterize the proposed \alg algorithm.
Section \ref{sec:simu} exploits \alg to solve the tracking problem (or dynamic consensus) over the maximum, median, and average functions, and learning problems with logistic loss function.
Section \ref{sec:conclusion} concludes the paper by discussing some future research directions.


\section{Open Operator Theory:\\ iterative behavior of operators\\ between spaces of different dimension}\label{sec:OOT}

The set of real and integer numbers are denoted by $\rea$ and $\int$, respectively, while $\reap$ and $\nat$ denote their restriction to positive entries.
Matrices are denoted by uppercase letters, vectors and scalars by lowercase letters, while sets and spaces are denoted by uppercase calligraphic letters.
The identity matrix is denoted by $I_n$, $n\in\nat$, while the vectors of ones and zeros are denoted by $\bone_n$ and $\bzero_n$; subscripts are omitted if clear from the context.
{Letting $\I\neq \emptyset$ be a finite set of labels, we adopt the non-standard yet intuitive notation $\rea^\I$ to denote a vector space of finite dimension equal to the number of elements of $\I$, and by $x\in\rea^\I$ we denote a vector with labeled components $x_i\in\rea$ where each label $i\in\I$ corresponds to a label in $\I$.}
We limit our discussion to finite-dimensional Euclidean normed spaces $(\rea^\I,\norm{\cdot}_2)$ and denote the distance between two points ${x,y\in\rea^\I}$ with the same labeled components by ${d:\rea^\I\times \rea^\I\rightarrow \rea_+}$, given by ${d(x,y) = \norm{x- y}_2}$.
In turn, the distance of a point $x\in\rea^\I$ from a set $\X\subseteq\rea^\I$ and the distance between two sets $\X,\Y\subseteq \rea^\I$ are defined as follows
\begin{equation}\label{eq:dists}
    d(x,\Y) = \inf_{y\in\Y} \ d(x,y), \quad d(\X,\Y) = \inf_{x \in \X} \ d(x,\Y).
\end{equation}

In the remainder of this section, we introduce all relevant concepts about operators mapping spaces $\X\subseteq\rea^{\I_1}$ into spaces $\Y\subseteq\rea^{\I_2}$ where $\I_1$, $\I_2$ are set of labels with possibly different dimension and elements, we call them \virg{\textit{open}} operators.

\begin{figure}[!b]
    \centering
    \includegraphics[width=0.57\linewidth]{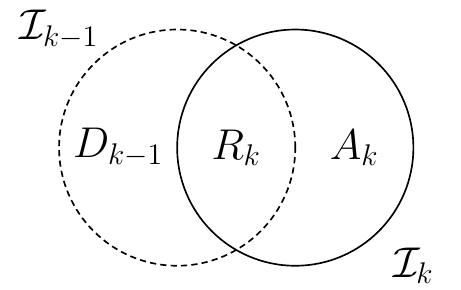}
    \caption{Venn diagram of labels at two consecutive steps $k-1$ and $k$.}
    \label{fig:venn}
\end{figure}

\subsection{Open-operators and sets of interest}

A time-varying \virg{\textit{open}} operator ${\Ts_k:\rea^{\I_{k-1}}\rightarrow \rea^{\I_k}}$ maps points with components labeled by $\I_{k-1}$ into points with components labeled by $\I_{k}$, which may have different size, yielding the following iteration:
\begin{equation}\label{eq:open_iter}
x_k = \Ts_k(x_{k-1}),\qquad x_k\in\rea^{\I_k}.    
\end{equation}
The sequence $\{x_k:k\in\nat\}$ generated by the iteration \eqref{eq:open_iter} is such that the points $x_k\in\rea^{\I_k}$ may have different dimensions for $k\in\nat$, and thus are called \virg{\textit{open sequences}}. According to Fig.~\ref{fig:venn}, we identify the following subsets:
\begin{itemize}
    \item \textit{Remaining labels} $\R_k=\I_{k}\cap \I_{k-1}$: labels present both at time $k-1$ and time $k$;
    \item \textit{Arriving labels} $\A_k=\I_{k}\setminus \I_{k-1}$: labels present at time $k$ but not at time $k-1$;
    \item \textit{Departing labels} $\D_k=\I_{k}\setminus\I_{k+1}$: labels present at time $k$ but not at time $k+1$.
\end{itemize}
By convention, $\I_{-1}=\emptyset$, yielding $\R_0=\emptyset$ and $\A_0=\I_0$.
Note that, in general, the set of departing labels may contain both remaining and arriving labels, which are instead disjoint:
$$
\D_k \subset \I_k = \R_k\cup \A_k,\qquad \R_k \cap \A_k = \emptyset.
$$
\begin{exmp}\label{exa:change_idx}
    Consider the sets of labels ${\I_0= \{a,b,c\}}$, $\I_1= \{a,b,c,d\}$, $\I_2= \{b,c,d\}$, $\I_3= \{e,f\}$. Then:
    \begingroup
    \medmuskip=2mu
    \thinmuskip=2mu
    \thickmuskip=2mu
    \begin{itemize}
        \item (at the initial step $0$) \hspace{0.2em} $D_{-1}=\emptyset$, $\A_0 = \{a,b,c\}$, $\R_0 = \emptyset$;
        \item (from step $0$ to step $1$) $D_{0}=\emptyset$, $\A_1 = \{d\}$, $\R_1 = \{a,b,c\}$;
        \item (from step $1$ to step $2$) $D_{1}=\{a\}$, $\A_2 = \emptyset$, $\R_2 = \{b,c,d\}$;
        \item (from step $2$ to step $3$) $D_{2}=\{b,c,d\}$, $\A_3 = \{e,f\}$, $\R_3 = \emptyset$.
    \end{itemize}
    \endgroup
    Accordingly, the points in the open sequence $\{x_0,x_1,x_2,x_3\}$ are described component-wise as follows:
    $$
    x_0 = \begin{bmatrix}
        x^a_0\\x^b_0\\x^c_0
    \end{bmatrix},\quad
    x_1 = \begin{bmatrix}
        x^a_1\\x^b_1\\x^c_1\\x^d_1
    \end{bmatrix},\quad
    x_2 = \begin{bmatrix}
        x^b_2\\x^c_2\\x^d_2
    \end{bmatrix},\quad
    x_3 = \begin{bmatrix}
        x^e_3\\x^f_3
    \end{bmatrix}.
    $$
\end{exmp}

Since $x_k$ takes values in $\rea^{\I_k}$ at all times $k$, components ${i\in \D_{k-1}}$ are simply left out from $x_{k}$. Instead, new components $i\in \A_k$ need to be initialized to some value $x^{\textsc{A},i}_k$ and remaining components $i\in \R^i_k$ need to be updated according to a scalar operator ${\Fs^i_k:\rea^{\I_{k-1}} \rightarrow \rea}$. 
Thus, the iteration in~\eqref{eq:open_iter} of the open operator $\Ts_k$ can be written component-wise for each $i\in \I_{k}$, as follows:
\begingroup
\medmuskip=2mu
\thinmuskip=2mu
\thickmuskip=2mu
\begin{equation}\label{eq:open_iter_cw}
x^{i}_k = \Ts^i_k(x_{k-1})=\begin{cases}
\Fs^i_k(x_{k-1}) & \text{if } i\in \R_k=\I_{k-1}\setminus \D_{k-1},\\
x^{\textsc{A},i}_k & \text{if } i\in \A_k=\I_k\setminus \I_{k-1}.
\end{cases}
\end{equation}
\endgroup
Note that if the components of the system remain unchanged, that is ${\I_{k}=\I_{k-1}}$, then the operator ${\Fs_k:\rea^{\I_{k-1}}\rightarrow \rea^{\I_{k-1}}}$ rules the so-called \virg{\textit{standard iteration}}:
\begingroup
\begin{equation}\label{eq:std_iter}
x_k=\Fs_k(x_{k-1}), \quad \text{when} \quad \I_k=\I_{k-1}.
\end{equation}
\endgroup

\begin{exmp}\label{exa:open_seq_lin}
    Consider the time-varying open operator as in~\eqref{eq:open_iter_cw} ruled by ${\Fs_k: x\mapsto F_k x}$ where
    $$
    F_k = \frac{1}{{\abs{\I_{k-1}}}}\bone_{\abs{\I_{k-1}}}\bone_{\abs{\I_{k-1}}}^\top.
    $$
    If the set of labels evolves as in Example~\ref{exa:change_idx}, then a possible open sequence generated by the iteration $x_k=\Ts_k(x_{k-1})$ is:
    \begingroup
    \medmuskip=0mu
    \thinmuskip=0mu
    \thickmuskip=0mu
    $$
    \underbrace{\begin{bmatrix}
        x^a_0\\x^b_0\\x^c_0
    \end{bmatrix}}_{x_0} = \begin{bmatrix}
        2\\3\\4
    \end{bmatrix},\quad
    \underbrace{\begin{bmatrix}
        x^a_1\\x^b_1\\x^c_1\\x^d_1
    \end{bmatrix}}_{x_1} = \begin{bmatrix}
        3\\3\\3\\7
    \end{bmatrix},\quad
    \underbrace{\begin{bmatrix}
        x^b_2\\x^c_2\\x^d_2
    \end{bmatrix}}_{x_2}=\begin{bmatrix}
        4\\4\\4
    \end{bmatrix},\quad
    \underbrace{\begin{bmatrix}
        x^e_3\\x^f_3
    \end{bmatrix}}_{x_3} = \begin{bmatrix}
        0\\1
    \end{bmatrix},
    $$
    \endgroup
    where:
    \begin{itemize}
    \item (from step $0$ to step $1$) all components associated with remaining labels update their value to the average of $x_0$, that is $(2+3+4)/3=3$; label $d$ arrives, and component $x^d_1$ is initialized at $1$;
    \item (from step $1$ to step $2$) all components associated with remaining labels update their value to the average of $x_1$, that is $(3+3+3+7)/4=4$; label $a$ departs, and component $x^a_2$ is left out;
    \item (from step $2$ to step $3$) labels $b,c,d$ departs, and components $x^b_3$, $x^c_3$, $x^d_3$ are left out; labels $e,f$ arrive, and components $x^e_3$, $x^f_3$ are initialized at $0$, $1$, respectively.
    \end{itemize}
\end{exmp}

The need to generalize the usual concept of \virg{\textit{set of fixed points}} $\hat{\X} = \{x : x=\Fs(x)\}$ for time-invariant standard operators $\Fs:\rea^{\I}\rightarrow\rea^{\I}$ arises naturally.
For open operators holds instead $\Fs_k:\rea^{\I_{k-1}}\rightarrow\rea^{\I_{k-1}}$ where $\I_k$ may change at any $k\in\nat$, yielding a trajectory of sets containing fixed points of different dimension, which we call the \virg{\textit{trajectory of sets of interest}}.
\begin{defn}[Trajectory of sets of interest]\label{def:TSI}
Consider~the iteration of an open operator $\Ts_k$ as in~\eqref{eq:open_iter} and let~$\hat{\X}_{k}$~be the set of fixed points of the operator $\Fs_{k+1}$ ruling the standard iteration as in~\eqref{eq:std_iter}, i.e.,
\begin{equation}\label{eq:tvg}
\hat{\X}_{k} := \{ x\in \rea^{\I_k} \mid  x = \Fs_{k+1}(x)\}.
\end{equation}
Then, the sequence ${\{\hat{\X}_k:k\in\nat\}}$ is called the \virg{trajectory of sets of interest} (TSI) of the operator $\Ts_k$. 
\end{defn}
If each operator $\Fs_k$ has a unique equilibrium point ${\hat{\X}_k = \{ \hat{x}_k \}}$, then the sequence $\{\hat{x}_k:k\in\nat\}$ of these points forms the \virg{\textit{trajectory of points of interest}}~\cite[Definition 3.1]{franceschelli2020stability}.

\begin{exmp}\label{exa:TSI}
    Consider the iteration described in Example~\ref{exa:open_seq_lin}. Then, the TSI is given by
    $$
    \hat{\X}_k = \{\alpha\bone\in \rea^{\I_{k}}\mid \alpha\in\rea\}.
    $$
    Indeed, at each $k\in\nat$, the set of fixed points of the standard iteration ruled by ${\Fs_k: x\mapsto F_k x}$ consists of the eigenvectors of $F_k$ corresponding to the (unique) unitary eigenvalue.
\end{exmp}

We formalize the concept of convergence for open sequences, emphasizing that comparing distances in spaces of different dimensions is unfair without normalization by the corresponding space's dimension.
Without normalization, iterations of an open operator with increasing number of components may produce sequences diverging from the TSI, even if the distance between new components and their TSI counterparts remains bounded.
We define convergence based on normalized distance from the TSI and provide an example clarifying the importance of this normalization.
\begin{defn}[Convergence of open sequences]\label{def:conv_os}
Consider the open sequence ${\{x_k\in\rea^{\I_k}:k\in\nat \}}$ generated by the iteration of an open operator as in~\eqref{eq:open_iter} whose TSI is ${\{\hat{\X}_k\subseteq \rea^{\I_k}:k\in\nat\}}$.
The open sequence is said to \virg{converge} to the TSI within a radius $R\geq0$ if
$$
\limsup_{k\rightarrow \infty} \ \frac{d(x_{k},\hat{\X}_k)}{\sqrt{\abs{\I_k}}}\leq R.
$$
\end{defn}

\begin{exmp}\label{exa:changing}
    Consider the iteration described in Example~\ref{exa:open_seq_lin} and assume that the components are totally renewed at each step and increase in number by 2 at each step, namely, ${D_k=\I_k}$ and ${\abs{\A_k}= \abs{\I_{k-1}} + 2}$.
    A possible open sequence generated by the iteration $x_k=\Ts_k(x_{k-1})$ is 
    \begingroup
    $$
    x_0 =
    \begin{bmatrix}
        +1\\-1
    \end{bmatrix}, \ \
    x_1 =
    \begin{bmatrix}
        +1\\+1\\-1\\-1
    \end{bmatrix}, \ \
    x_2 = 
    \begin{bmatrix}
        +1\\+1\\+1\\-1\\-1\\-1
    \end{bmatrix}, \ \ \cdots \ ,
    $$
    \endgroup
    where half of the components are initialized at $+1$ and the other half at $-1$.
    The point $\hat{x}_k\in\hat{\X}_k$ in the TSI attaining the minimum distance from $x_k$ is the null vector $\hat{x}_k=\bzero_{2(k+1)}$ and therefore 
    $$
    d(x_k,\hat{\X}_k) = d(x_k,\hat{x}_k) = d(x_k,\bzero_{2(k+1)}) = \sqrt{2(k+1)},
    $$
    which, in the limit of $k\rightarrow \infty$, diverges even though the new components have bounded distance of $1$ from the corresponding component of the TSI.
    Instead, when normalization is considered, it is possible to find a finite upper bound
    $$
    \limsup_{k\rightarrow\infty} \frac{d(x_k,\hat{\X}_k)}{\sqrt{2(k+1)}} = 1.
    $$
\end{exmp}

\subsection{Open distances and projections}
A notion of \virg{\textit{open distance}} is necessary to evaluate the distance between points with labeled components that belong to vector spaces of different dimensions.
We propose to evaluate such distance by only taking into consideration the common labeled components, disregarding the others. 
\begin{defn}[Open distance]
    Let $\I_1$ and $\I_2$ be two finite sets of labels. If $\I_1\cap \I_2\neq \emptyset$, the open distance ${d:\rea^{\I_1}\times \rea^{\I_2}\mapsto \rea_+}$ between points $x\in\rea^{\I_1}$ and $y\in\rea^{\I_2}$ is defined by
    $$
    d(x,y) = \norm{\tilde{x}-\tilde{y}}_2,\quad \text{where} \quad \begin{cases}
        \tilde{x}&=[x^i \text{ for } i\in\I_1\cap \I_2],\\
        \tilde{y}&=[y^i \text{ for } i\in\I_1\cap \I_2].
    \end{cases}
    $$
    Otherwise, if $\I_1\cap \I_2= \emptyset$ then $d(x,y)=0$.
\end{defn}
\begin{exmp}
    Consider the sets of labels $\I_1=\{a,b,c\}$ and $\I_2=\{b,c,d\}$ and let
    $$
    x=\begin{bmatrix}
        x^a\\x^b\\x^c
    \end{bmatrix}
    =
    \begin{bmatrix}
        1\\2\\3
    \end{bmatrix}\in \rea^{\I_1},\qquad
    y=\begin{bmatrix}
        y^b\\y^c\\y^d
    \end{bmatrix}
    =
    \begin{bmatrix}
        4\\5\\6
    \end{bmatrix}\in \rea^{\I_1}.
    $$
    Since the common components are those labeled by $b$ and $c$, the distance between them is given by
    $$
    d(x,y) = \norm{\begin{bmatrix}
        x^b\\x^c
    \end{bmatrix} - \begin{bmatrix}
        y^b\\y^c
    \end{bmatrix}} = \norm{\begin{bmatrix}
        2\\3
    \end{bmatrix} - \begin{bmatrix}
        4\\5
    \end{bmatrix}} = \sqrt{8}.
    $$
\end{exmp}

For the sake of simplicity, we overload the notation of standard distance in a way that the distance of a point $x\in\rea^{\I_1}$ from a set $\X\subseteq \rea^{\I_2}$ and the distance between two sets $\X\subseteq \rea^{\I_2}$, $\Y\subseteq \rea^{\I_3}$ of different dimension are naturally defined as in~\eqref{eq:dists}.
The projection of a point $x\in\rea^{\I_1}$ over a non-empty set ${\Y\subseteq \rea^{\I_2}}$ is, in general, a set of points $y\in \Y$ of minimum distance from $x$, given by the \textit{projection operator} defined next
$$
\proj(x,\Y) = \left\{y^\star\in\Y :\: d(x,y^\star)=\inf_{y\in\Y} \ d(x,y) \right\}
$$
Note that if $\I_1\supseteq\I_2$ then the projection reduces to a singleton (see Fig.~\ref{fig:larger_to_lower}), but if $\I_2\setminus \I_1\neq 0$ then it is indeed a set (see Fig.~\ref{fig:lower_to_larger}) because the components related to the elements in $\I_2$ that are not in $\I_1$ can be arbitrary within $\Y$. The same holds for the proximal.

\begin{exmp}
    Consider the sets of labels $\I_1=\{a,b\}$ and $\I_2=\{a\}$ as in Fig.~\ref{fig:larger_to_lower} and let $\Y=\{y\in\rea^{\I_2}\mid y\in[1,3]\}$. Then, considering the representation on the left of Fig.~\ref{fig:larger_to_lower}, the distance of $x=[x^a, x^b]^\top = [5, 1]^\top\in\rea^{\I_1}$ from $\Y$ is
    $$
    d(x,\Y) = d(x^a,\Y) = \min_{y\in\Y} \abs{5-y} = \abs{5-3} = 2, 
    $$
    and the projection of $x$ into $\Y$ is a singleton given by
    $$
    \proj(x,\Y) = 3\in\rea^{\I_2}.
    $$
    Instead, considering the representation on the right of Fig.~\ref{fig:larger_to_lower}, the distance of $x=[x^a, x^b]^\top = [2, 1]^\top\in\rea^{\I_1}$ from $\Y$ is
    $$
    d(x,\Y) = d(x^a,\Y) = \min_{y\in\Y} \abs{2-y} = 0, 
    $$
    and the projection of $x$ into $\Y$ is a singleton given by
    $$
    \proj(x,\Y) = 2\in\rea^{\I_2}.
    $$
\end{exmp}
\vspace{-1em}
\begin{figure}[!h]
    \centering
    \includegraphics[width=0.98\linewidth]{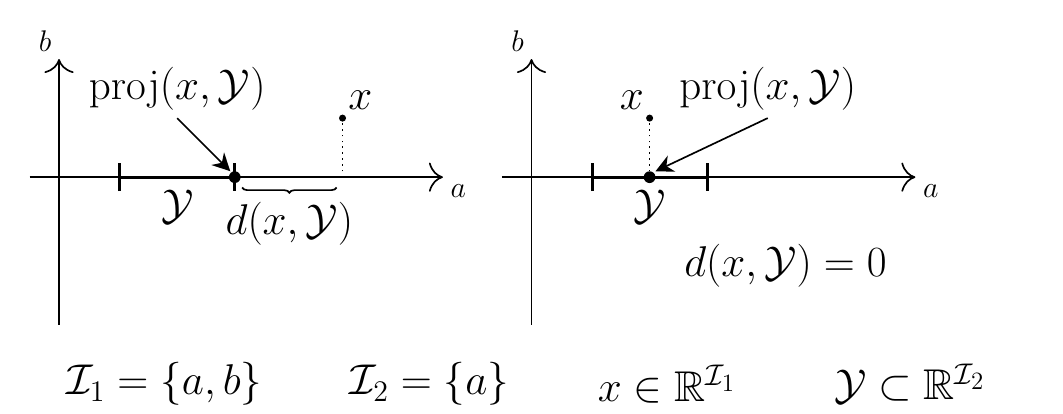}
    \caption{Representation of the open distance (curly brace) and the open projection (bold dot) in two examples with ${\mathcal{I}_2\subseteq \I_1\neq \emptyset}$. }
    \label{fig:larger_to_lower}
\end{figure}

\begin{exmp}
    Consider the sets of labels $\I_1=\{a\}$ and $\I_2=\{a,b\}$ as in Fig.~\ref{fig:lower_to_larger}.
    Then, considering the representation on the left of Fig.~\ref{fig:lower_to_larger}, the distance of $x= 5\in\rea^{\I_1}$ from $\Y=\{[y^a, y^b]\in\rea^{\I_2}\mid y^a\in[1,2],\ y^b\in\rea\}$ is
    $$
    d(x,\Y) = \min_{[y^a,y^b]\in\Y} \abs{5-y^a} = \abs{5-2} = 3, 
    $$
    and the projection of $x$ into $\Y$ is a set given by
    $$
    \proj(x,\Y) = \{[z^a,z^b]^\top\in\rea^{\I_2} \mid  z^a=2, \ z^b\in\rea\}.
    $$
    Instead, considering the representation on the right of Fig.~\ref{fig:larger_to_lower}, the distance of $x=5\in\rea^{\I_1}$ from $\Y=\{[y^a, y^b]\in\rea^{\I_2}\mid y^a\in\rea,\ y^b\in[-2,-1\}$ is
    $$
    d(x,\Y) = d(x^a,\Y) = \min_{[y^a,y^b]\in\Y} \abs{2-y^a} = 0, 
    $$
    and the projection of $x$ into $\Y$ is a set given by
    $$
    \proj(x,\Y) = \{[z^a,z^b]^\top\in\rea^{\I_2} \mid  z^a=5, \ z^b\in[-2,-1]\}.
    $$
\end{exmp}
\vspace{-1em}
\begin{figure}[!h]
    \centering
    \includegraphics[width=0.98\linewidth]{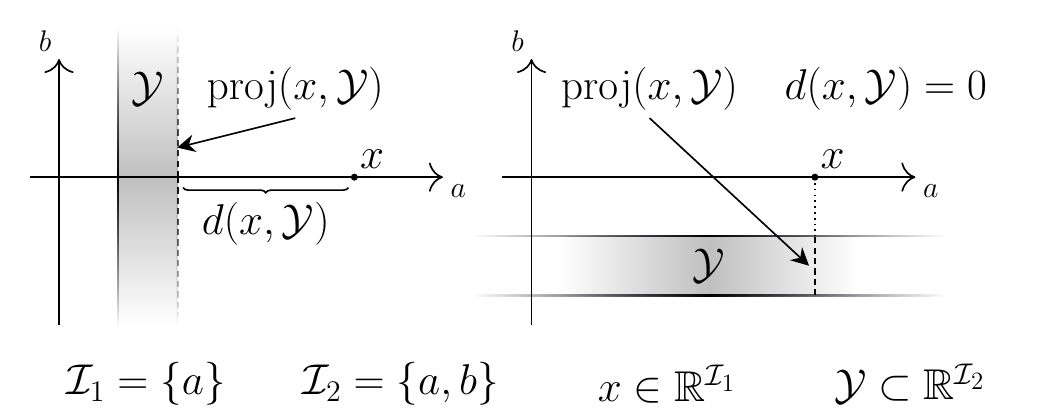}
    \caption{Representation of the open distance (curly brace) and the open projection (dashed line) in two examples with ${\mathcal{I}_2\setminus \I_1\neq \emptyset}$. }
    \label{fig:lower_to_larger}
\end{figure}

We also introduce the concept of \textit{shadow distance} between two spaces $\X\subseteq \rea^{\I_1}$, $\Y\subseteq\rea^{\I_2}$, that is the maximum distance between any pair of projections, namely,
$$
\dsh (\X,\Y) = \sup_{z\in\rea^{\I_1\cup\I_2}} d(\proj(z,\X),\proj(z,\Y)).
$$
The concept of shadow distance is essential as it allows formulating the following version of triangle inequality:
$$
d(z,\X)\leq d(z,\Y) + \dsh(\X,\Y).
$$

\begin{exmp}
    Consider the set of labels $\I=\{a\}$ and the sets $\X=\{x^{\I}\in\rea:x\in[1,2]\}$, $\Y=\{y^{\I}\in\rea:y\in[5,6]\}$. The standard distance between these sets is attained by the points $x=2$ and $y=5$, namely,
    $$
       d(\X,\Y) = \inf_{x \in \X} \ \inf_{y\in\Y} \ d(x,y) = \abs{2-5} = 3.
    $$
    On the other hand, the shadow distance (see Fig. \ref{fig:shadow}) is attained for all points $z\in\rea^{\I}$ such that $z\in(-\infty,2]\cup [6,+\infty)$. For $z\leq 2$, its projection onto $\X$ is $1$ and its projection onto $\Y$ is $5$, yielding
    $$
    \dsh (\X,\Y) = \sup_{z\in\rea^{\I}} d(\proj(z,\X),\proj(z,\Y)) = \abs{1-5} = 4.
    $$
\end{exmp}

\begin{figure}[!h]
    \centering
\includegraphics[width=0.98\linewidth]{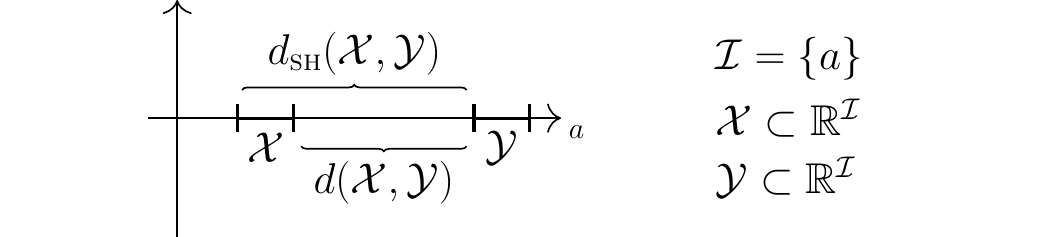}
\caption{Representation of the shadow distance compared to the standard distance in an example with ${\I=\mathcal{I}_1 = \I_2}$.}
    \label{fig:shadow}\vspace{-1em}
\end{figure}

\subsection{A convergence result for paracontractive operators}
In this section, we provide sufficient conditions ensuring the convergence of open sequences generated by the iteration of open operators  -- in the sense of Definition~\ref{def:conv_os} -- whose standard iteration is ruled by a \virg{\textit{paracontractive}} operator~\cite{elsner1992convergence,raja2021payoff,rai2023distributed}. {Paracontractive operators are a class of operators that generalizes that of contractive operators, as they enjoy a contractivity property between trajectories and fixed points, rather than between trajectories in general.
An alternative name could be \virg{quasicontractive} operators, following the definition of \virg{quasinonexpansive} operators given by Bauschke and Combettes in their book \cite[Page 70]{Bauschke2017}.}
\begin{defn}
An operator ${\Fs_k:\rea^{\I}\rightarrow\rea^{\I}}$ is said to be \virg{paracontractive} if there exists $\gamma\in[0,1)$ such that for all $k\geq 0$ and for all $x\in\rea^{\I}$ it holds
\begin{equation}\label{eq:paracon}
d(\Fs_{k+1}(x),\hat{\X}_{k}) \leq \gamma \cdot d(x,\hat{\X}_{k}),
\end{equation}
where $\hat{\X}_k$ is the set of fixed points of $\Fs_{k+1}$ for all $k\in\nat$. 
\end{defn}

The sufficient conditions we enforce correspond to limits on the variation of the TSI and on the process by which the labels components join and leave during the iteration.
In plain words, the normalized shadow distance between two consecutive sets of interests and the normalized distance of the arriving labeled components from the set of interest must be bounded from above. Moreover, the dimension of the state space cannot decrease too fast, i.e., there must be an upper bound to the ratio between the number of components in two consecutive steps.
These limits are formally defined next.
\begin{defn}[Bounded TSI]
Consider the TSI is ${\{\hat{\X}_k\subseteq \rea^{\I_k}:k\in\nat\}}$ of an open operator as in~\eqref{eq:open_iter}.
The TSI is said to have \virg{bounded variation} if
$$
\exists B\geq 0:\quad \frac{\dsh (\hat{\X}_{k},\hat{\X}_{k-1})}{\sqrt{\abs{\R_k}}}\leq B,\quad \forall k\in\nat.
$$
\end{defn}
\begin{defn}[Bounded departure process]\label{def:bounded-departure}
Consider the open sequence ${\{x_k\in\rea^{\I_k}:k\in\nat \}}$ generated by the iteration of an open operator as in~\eqref{eq:open_iter}.
The departure process is said to be \virg{bounded}~if
$$
\exists \beta\in(0,1):\quad \sqrt{\abs{\I_k}}\geq \beta \sqrt{\abs{\I_{k-1}}},\quad \forall k\in\nat.
$$
\end{defn}
\begin{defn}[Bounded arrival process]\label{def:bounded-arrival}
Consider the open sequence ${\{x_k\in\rea^{\I_k}:k\in\nat \}}$ generated by the iteration of an open operator as in~\eqref{eq:open_iter} whose TSI is ${\{\hat{\X}_k\subseteq \rea^{\I_k}:k\in\nat\}}$.
Let $x_k^{\textsc{A}}\in \rea^{\A_k}$ be the vector stacking the components of all arriving labels.
The arrival process is said to be \virg{bounded}~if
$$
\exists H\geq 0:\quad \frac{d(x^{\textsc{A}}_k,\hat{\X}_{k})}{\sqrt{\abs{\A_k}}}\leq H,\quad \forall k\in\nat.
$$
\end{defn}

\begin{exmp}
    Consider the iteration described in Example~\ref{exa:open_seq_lin} whose TSI is that described in Example~\ref{exa:TSI} and assume that the components change as outlined in Example~\ref{exa:changing}. Then:
    \begin{itemize}
        \item the TSI is bounded with $B=0$ because there are never remaining components;
        \item the departure process is bounded for any $\beta\in[0,1]$ because the number of components never decreases;
        \item the arrival process is bounded with $H=1$ because all components are arriving components initialized at $\{-1,+1\}$ and the corresponding components of the TSI are zeros. 
    \end{itemize} 
\end{exmp}

We now state and prove the main result of this section.

\begin{thm}\label{th:mainth}
Consider the iteration of a time-varying open operator $\Ts_k:\rea^{\I_{k-1}}\rightarrow \rea^{\I_k}$ given component-wise for $i\in\I_k$ by
\begingroup
\medmuskip=2mu
\thinmuskip=2mu
\thickmuskip=2mu
$$
x^{i}_k = \Ts^i_k(x_{k-1})=\begin{cases}
\Fs^i_k(x_{k-1}) & \text{if } i\in \R_k=\I_{k-1}\setminus \D_{k-1},\\
x^{\textsc{A},i}_k & \text{if } i\in \A_k=\I_k\setminus \I_{k-1}.
\end{cases}
$$
\endgroup
and assume that
\begin{itemize}
    \item[(a)] $\Fs_k$ is paracontractive with $\gamma\in(0,1)$;
    \item[(b)] the TSI has bounded variation $B\geq 0$;
    \item[(c)] the arrival process is bounded with $H\geq 0$.
    \item[(d)] the departure process is bounded with $\beta\in(\gamma,1)$.
\end{itemize}
Then, the open sequence ${\{x_k\in\rea^{\I_k}:k\in\nat \}}$ converges linearly with rate $\theta=\gamma/\beta\in(0,1)$ to the TSI within a radius
\begin{equation}\label{eq:radius}
    R = \dfrac{B+H}{1-\theta},
\end{equation}
according to the following punctual upper bound
$$
\frac{d(x_{k},\hat{\X}_k)}{\sqrt{\abs{\I_k}}} \leq \theta^k \frac{d(x_{0},\hat{\X}_0)}{\sqrt{\abs{\I_0}}} + \frac{1-\theta^k}{1-\theta}(B+H).
$$
\end{thm}

\begin{proof}
Let us split the state $x_k=[x^\textsc{R}_k;x^\textsc{A}_k]$ into two vectors such that:
\begin{itemize}
    \item $x_k^{\textsc{R}}\in \rea^{\R_k}$ is the vector of the remaining components; 
    \item $x_k^{\textsc{A}}\in \rea^{\A_k}$ is the vector of the new components.
\end{itemize}
For any two consecutive steps $k-1,k$ with $k\in\nat$, it holds:
\allowdisplaybreaks
\begin{align*}
d(x_{k},\hat{\X}_k) & = \sqrt{d^2(x_k^{\textsc{R}},\hat{\X}_k)+ d^2(x_k^{\textsc{A}},\hat{\X}_k)},\\
& \leq \sqrt{d^2(x_k^{\textsc{R}},\hat{\X}_k)}+\sqrt{ d^2(x_k^{\textsc{A}},\hat{\X}_k)},\\
& = d(x_k^{\textsc{R}},\hat{\X}_k)+ d(x_k^{\textsc{A}},\hat{\X}_k),\\
&\overset{(i)}{\leq}  d(x_k^{\textsc{R}},\hat{\X}_k)+ \sqrt{\abs{\I_k}}H,\\
&\overset{(ii)}{\leq} d(x_k^{\textsc{R}},\hat{\X}_{k-1})+ \dsh(\hat{\X}_k,\hat{\X}_{k-1}) + \sqrt{\abs{\I_k}}H,\\
%
&\overset{(iii)}{\leq} d(x_k^{\textsc{R}},\hat{\X}_{k-1})+ \sqrt{\abs{\I_k}}(B + H),\\
%
&\overset{(iv)}{\leq} d(\Fs_k(x_{k-1}),\hat{\X}_{k-1})+ \sqrt{\abs{\I_k}}(B + H),\\
&\overset{(v)}{\leq} \gamma d(x_{k-1},\hat{\X}_{k-1}) + \sqrt{\abs{\I_k}}(B + H),
\end{align*}
where $(i)$ holds by assumption $(c)$ and $\abs{\A_k}\leq \abs{\I_k}$; $(ii)$ holds by triangle inequality; $(iii)$ holds by assumption $(b)$ and $\abs{\R_k}\leq \abs{\I_k}$; $(iv)$ holds because the vector $x_k^{\textsc{R}}$ is entirely contained into $\Fs_k(x_{k-1})$ by definition; $(v)$ holds by assumption $(a)$.
Thus, we have shown that,
$$
    \frac{d(x_{k},\hat{\X}_k)}{\sqrt{\abs{\I_k}}} \leq \gamma\frac{d(x_{k-1},\hat{\X}_{k-1})}{\sqrt{\abs{\I_k}}}  + B + H.
$$
which, by assumption $(d)$, becomes
$$
\frac{d(x_{k},\hat{\X}_k)}{\sqrt{\abs{\I_k}}} \leq \frac{\gamma}{\beta}\frac{d(x_{k-1},\hat{\X}_{k-1})}{\sqrt{\abs{\I_{k-1}}}}  + B + H,
$$
By iterating over $k\in\nat$ one gets
\begingroup
$$
\frac{d(x_{k},\hat{\X}_k)}{\sqrt{\abs{\I_k}}} \leq \left(\frac{\gamma}{\beta}\right)^k \frac{d(x_{0},\hat{\X}_0)}{\sqrt{\abs{\I_0}}} + (B+H)\sum_{i=0}^{k-1} \left(\frac{\gamma}{\beta}\right)^i.
$$
\endgroup
Since in the limit of $k\rightarrow \infty$, then $(\gamma/\beta)^k$ goes to $0$ and the geometric series equals to $(1-\theta^k)/(1-\theta)$ and goes to $1/(1-\gamma/\beta)$, it holds
$$
\limsup_{k\rightarrow \infty} \frac{d(x_{k},\hat{\X}_k)}{\sqrt{\abs{\I_k}}}  \leq \frac{B+H}{1-\frac{\gamma}{\beta}} =: R,
$$
thus concluding the proof.
\end{proof}
\begin{cor}\label{cor:mainth}
Consider the scenario of Theorem~\ref{th:mainth} and the following simplified cases:
\begin{itemize}
    \item[$(a)$] The iteration is not open;
    \item[$(b)$] The iteration is time-invariant.
\end{itemize}
Then, the results of Theorem~\ref{th:mainth} become:
\begin{itemize}
    \item $(a)$ implies that the convergence rate is equal to the paracontractivity constant $\theta=\gamma$ because $\beta=1$ and implies that the radius reduces to $R=B/(1-\gamma)$ because $H=0$;
    \item $(a)\land(b)$ implies that the radius reduces to zero because $B=0$, i.e., the sequence converges to a fixed point for any initial condition.
\end{itemize}
\end{cor}

\newpage

\section{Problem of Interest and\\ Proposed Algorithm: \alg}\label{sec:algorithm}

With the theoretical framework designed in Section \ref{sec:OOT}, we have now the tools to design and analyze algorithms for optimization and learning over open networks. In particular, we are interested in algorithms that can solve the following
\begin{equation}\label{eq:optprob}
    \min_{y\in\rea^\P} \sum_{i \in\V_k} f^i_k(y),
\end{equation}
where $\P$ is a set of labels, $f^i_k:\rea^\P\mapsto\rea$ denotes the local objective function of an agent $i\in\V_k$ in the network at time $k$, where $\V_k$ represents the time-varying set of agents.
Most of the literature addresses this problem under the assumption that the set of agents in the network remains constant over time, i.e., ${\V=\V_0=\V_1,\cdots}$, which facilitates the application of various results from operator theory to study the convergence of custom-designed algorithms.
Instead, this work considers the problem without this assumption, thus letting the agents be able to leave and join the network arbitrarily, which results in a time-varying number of agents ${n_k=\abs{\V_k}\in(0,\infty)}$: such networks are usually called \virg{\textit{Open Multi-Agent Systems}} (OMASs).
We propose a distributed algorithm called \virg{\alg} to solve the problem in~\eqref{eq:optprob} in OMASs and we carry out a convergence analysis by means of \virg{\textit{open operator theory}} developed in the previous Section~\ref{sec:OOT}.

\begin{rem}[Relationship with online optimization]
We remark that~\eqref{eq:optprob} is an online optimization problem because the optimal solution is time-varying due to two different factors: 1) the local costs are time-varying \cite{li_survey_2023}; 2) the agents participating in the network change over time, thus yielding a change in the set of the local costs. Consequently, even when the local costs are static, the problem solution is still time-varying due to the open nature of the network.
\end{rem}

\subsection{Problem set-up and working assumptions}
An OMAS consists of a time-varying set of agents $\V_k$ which may leave and join the network at any time $k\in\nat$. We define the sets of remaining/arriving/departing agents as follows
$$
\Vr_k=\V_{k}\cap V_{k-1},\:\:
\Va_k=\V_{k}\setminus \V_{k-1},\:\:
\Vd_k=\V_{k}\setminus\V_{k+1},
$$
The agents are linked according to a graph ${\G_k=(\V_k,\E_k)}$, where ${\E_k\subseteq \V_k\times \V_k}$ represents the set of agents' pairs that are linked by a point-to-point communication channel.
The set of agents that can communicate with the $i$-th agent at time $k$ is denoted by $\mathcal{N}^i_{k}=\{j\in \V_k: (i,j)\in \E_k\}$ and its cardinality is denoted by $\eta^i_{k}=\abs{\mathcal{N}^i_{k}}$; note that graphs are assumed to be without self-loops, i.e., $i\notin \mathcal{N}^i_{k}$.
We also denote by ${\xi_k=\abs{\E_k}=\eta^1_k+\cdots+\eta^{n_k}_k}$ the total number of communication channels.
We formalize next our assumptions on the communication graph among the agents.
\begin{assum}\label{as:graph}
    The communication graph $\G_k=(\V_k,\E_k)$ of the OMAS satisfies the following at all times $k\in\nat$:
    \begin{itemize}
        \item undirected, i.e., $(i,j)\in\E_k$ if and only if $(j,i)\in \E_k$;
        \item connected, i.e., there is a sequence of consecutive pairs $(i,a),(a,b),\cdots,(y,z),(z,j)$ in $\E_k$ for all ${i,j\in\V_k}$.
    \end{itemize}
\end{assum}

An OMAS can actually solve the under suitable assumptions of the local objective functions and the corresponding local and global solutions. 
We formalize our set of assumptions in Assumption~\ref{as:all}, which makes use of the following notation for the set of solutions to the problem in~\eqref{eq:optprob},\vspace{-0.5em}
$$
\Y^{\star}_k = \left\{y^\star_{k}\in\rea^{\P}: \sum_{i \in\V_k} f^i_k(y^\star_{k})=\min_{y\in\rea^{\P}} \sum_{i \in\V_k} f^i_k(y)\right\},
$$
and for the minimizers of each local cost function,\vspace{-0.5em}
\begin{equation}\label{eq:opt-loc}
\Y^{i,\star}_k = \left\{y^{i,\star}_{k}\in\rea^{\P}: f^i_k(y^{i,\star}_{k})=\min_{y\in\rea^{\P}} f^i_k(y)\right\}.
\end{equation}

\begin{assum}\label{as:all}
The problem in~\eqref{eq:optprob} is such that, $\forall k\in\nat $:
\begin{itemize}
    \item[(i)] the local cost functions $f^i_k$ are proper~\cite[Definition~1.4]{Bauschke2017}, lower semi-continuous~\cite[Lemma~1.24]{Bauschke2017}, and convex~\cite[Definition~8.1]{Bauschke2017} for all $i\in\V_k$;
    \item[(ii)] the set of minimizers $\Y^{i,\star}_k\subseteq \rea^p$ for each local cost function $f^i_k$ is not empty;
    \item[(iii)] the distance between two consecutive global solutions $y^\star_{k}\in\Y^\star_k$ and $y^\star_{k-1}\in\Y^\star_{k-1}$ is upper bounded by a constant $\sigma \geq 0$; 
    \item[(iv)] the distance between any local solution $y^{i,\star}_{k}\in \Y^{i,\star}_k$ and any global solution $y^\star_{k}\in \Y^\star_k$ is upper bounded by $\omega\geq 0$.
\end{itemize}
\end{assum}
Assumption~\ref{as:all}(i) is standard in the context of distributed optimization, but it is not sufficient to guarantee that the set of solutions $\Y^\star_k$ is not empty. This is, instead, ensured by Assumption~\ref{as:all}(ii), which requires that each local set of minimizers $\Y^{i,\star}_k$ is not empty.
On the other hand, $\Y^\star_k$ may also contain an infinite number of solutions and be unbounded. Thus, Assumption~\ref{as:all}(iii) ensures an upper bound on the distance between any solution at time $k$ and any solution at time $k-1$.
Due the heterogeneity of the local objective functions, the local minimizers in $\Y^{i,\star}_k$ can be arbitrarily far away from the global minimizers in $\Y^\star_k$. Thus, Assumption~\ref{as:all}(iv) ensures that such distance is bounded by a constant at any time $k$.\vspace{-1em}

\subsection{\alg and convergence analysis}\label{sec:DOOT-admm}
To solve the problem in~\eqref{eq:optprob} in a distributed way over an open network of agents, we propose the open version of the Alternating Direction Method of Multipliers (ADMM), which we call \virg{$\text{\alg}$}, whose implementation is detailed in Algorithm~\ref{alg:doot} provided in the next page.
\alg requires each agent $i\in\V_k$ to update/initialize a state variable $x^{ij}\in\rea^{\P}$ for every agent $j\in\N^i_k$ with which it has an open communication channel. By further defining,
$$
\R^i_k=\N^i_{k}\cap \N^i_{k-1},\:\: \A^i_k=\N^i_{k}\setminus \N^i_{k-1},\:\: \D^i_k=\N^i_{k}\setminus\N^i_{k+1},
$$
we can formalize the open operator describing \alg as follows, which makes use of two design parameters, the relaxation $\alpha\in(0,1)$ and the penalty $\rho>0$,
\begingroup
\medmuskip=2mu
\thinmuskip=2mu
\thickmuskip=2mu
$$
x^{ij}_k = \Ts^{ij}_k(x_{k-1})=\begin{cases}
\Fs^{ij}_k(x_{k-1}) & \text{if } i\in \Vr_{k} \land \ j\in \R^i_k,\\
\rho y^{i,\star}_k & \text{if } i\in \Va_k \land j\in\N^i_k,
\end{cases}
$$
\endgroup
where $y^{i,\star}_k\in \Y^{i,\star}_k$ is a local minimizer,
and
$$
\begin{aligned}
\Fs^{ij}_k(x_{k-1}) & = (1 - \alpha)  x^{ij}_{k-1}  + \alpha (2\rho y^j_{k-1} - x^{ji}_{k-1}),\\
y^j_{k-1} & = \prox_{f^j_{k-1}}^{1/\rho \eta^j_{k-1}} \Big( \textstyle \frac{1}{\rho \eta^j_{k-1}} \sum_{\ell \in \N^j_{k-1}}  x^{j\ell}_{k-1} \Big).
\end{aligned}
$$

The local update $x^{ij}_k=\Fs^{ij}_k(x_{k-1})$ requires that all agents ${j\in\R^i_k}$ have previously transmitted to agent $i\in\Vr_k$ both their state $x^{ji}_{k-1}$ and the quantity resulting by the proximal operation, which are internal variables of \alg denoted by $y^j_{k-1}$.
Since the standard iteration of \alg is ruled by the operator $\Fs_k=[\cdots,\Fs^{ij}_k,\cdots]^\top$, which is derived by applying the relaxed Peaceman-Rachford splitting method to the dual of the distributed version of the problem in~\eqref{eq:optprob} (see~\cite{Bastianello21} and references therein), its fixed points are such that all the internal variables are the same and equal to an optimal solution $y^\star_k\in\Y^\star_k$ of the problem in~\eqref{eq:optprob}. Thus, the internal variables constitute the output estimation of the agents of the optimal solutions.
The above described operations to be performed by each agent in the network to correctly execute \alg are detailed in Algorithm~\ref{alg:doot}.
To prove the convergence of \alg, we resort to our main convergence result in Theorem~\ref{th:mainth} for open operators. We first show that Assumptions~\ref{as:graph}-\ref{as:all} are sufficient to guarantee the existence of a TSI and its boundedness (Lemma~\ref{lem:DOOT_TSI}) as well as the boundedness of the arrival process (Lemma~\ref{lem:DOOT_BAR}). 

\begin{lem}\label{lem:DOOT_TSI}
Consider an OMAS running \alg to distributedly solve an optimization problem as in~\eqref{eq:optprob} under Assumptions~\ref{as:graph}-\ref{as:all}.
Then, there is a TSI $\{\hat{\X}_k:k\in\nat\}$ given by
\begingroup
\medmuskip=2mu
\thinmuskip=2mu
\thickmuskip=2mu
\begin{equation}\label{eq:TSI_existence}
\hat{\X}_{k} = \{ \hat{x}_k \mid (I+P_k)  \hat{x}_k = 2\rho P_kA_k(\bone_{n_k}\otimes y^\star_{k}),\ y^\star_{k}\in\Y^\star_k\},
\end{equation}
\endgroup
where $\Y^\star_k\neq \emptyset$ is the set of solutions to the problem in~\eqref{eq:optprob}. Moreover, the TSI is bounded with\vspace{-0.5em}
\begin{equation}\label{eq:TSI_bound}
B=\rho\sigma.
\end{equation}
\end{lem}\vspace{-0.5em}
\begin{proof}
    See Section A in the Appendix.
\end{proof}

\begin{lem}\label{lem:DOOT_BAR}
Consider an OMAS running \alg to distributedly solve an optimization problem as in~\eqref{eq:optprob} under Assumptions~\ref{as:graph}-\ref{as:all}.
Then, the arrival process is bounded with\vspace{-0.5em}
\begin{equation}\label{eq:TSI_boundarr}
H = \rho\omega. 
\end{equation}
\end{lem}\vspace{-0.5em}
\begin{proof}
    See Section B in the Appendix.
\end{proof}

\begin{algorithm}[!t]
\caption{ Open and distributed ADMM}
\label{alg:doot}
\begin{algorithmic}
\Require{The relaxation $\alpha \in (0, 1)$ and the penalty $\rho > 0$.}
\Ensure{Each agent returns $y^i_k$ that is an (approximated) solution to the optimization problem in~\eqref{eq:optprob}.}
\NoDo
\State \hspace{-1.3em} \textbf{for $k=0,1,2,\ldots$ each agent $i\in\V_k$:}
\If{$i\in \Va_{k}$ \textbf{is an arriving agent:}}
\State initializes the state variables to a local optimum
\begin{equation}\label{eq:init-x}
    x^{ij}_k = \rho y^{i,\star}_k \text{ (see \eqref{eq:opt-loc})}\qquad \forall j\in\N^i_k
\end{equation}
%
%
%
\ElsIf{$i\in\Vr_{k}$ \textbf{is a remaining agent:}}
\State{receives $y^j_{k-1}$, $x^{j,i}_{k-1}$ from each neighbor $j\in\R^i_{k}$}
\State updates the remaining state variable according to
\begingroup
\medmuskip=2mu
\thinmuskip=2mu
\thickmuskip=2mu
\begin{equation}\label{eq:admm-x}
     x^{ij}_k =  (1 - \alpha)  x^{ij}_{k-1} - \alpha  x^{ji}_{k-1} + 2\rho \alpha y^j_{k-1} \qquad \forall j\in\R^i_{k}
\end{equation}
\endgroup
\State initializes the new state variables to a local optimum
\begin{equation}\label{eq:init-x-2}
     x^{ij}_k = \displaystyle \rho y^{i,\star}_k \text{ (see \eqref{eq:opt-loc})}\qquad \forall j\in\A^i_{k}
\end{equation}
\EndIf
\State updates the output variable 
\begin{equation}\label{eq:admm-y}
y^i_{k} = \prox_{f^i_k}^{1/\rho \eta^i_k} \left( \frac{1}{\rho \eta^i_k} \sum_{j \in \N^i_{k}}  x^{ij}_{k} \right)
\end{equation}
\State transmits $y^i_k$, $x^{ij}_{k}$ to each neighbor $j\in\N^i_{k}$
\\
\hspace{-1em}\textbf{end for}
\end{algorithmic}
\end{algorithm}

\smallskip 

These intermediate lemmas allow proving convergence the convergence of \alg. In particular, Theorem~\ref{th:mainth_omas} (building on Theorem~\ref{th:mainth}) demonstrates that \alg is stable in the sense that the open sequences it generates converge within a radius of the set of interest.
After proving stability, we analyze the performance of \alg in Theorem~\ref{th:mainth_omas_opt}, where we prove that the open sequences $\{ y_k \ : \ k \in \nat \}$ converge linearly within a time-varying radius of the optimal solution sets $\{ \mathcal{C}_k^\star \ : \ k \in \nat \}$.
Finally, we provide Corollary~\ref{cor:mainth_omas} to characterize the performance in simplified scenarios.

\begin{thm}\label{th:mainth_omas}
    Consider an OMAS running \alg to distributedly solve an optimization problem as in~\eqref{eq:optprob} under Assumptions~\ref{as:graph}-\ref{as:all}.
    If the standard iteration is paracontractive with $\gamma\in (0,1)$ and the departure process is bounded with $\beta\in(\gamma,1)$, then the open sequence ${\{x_k:k\in\nat \}}$ generated by the open operator of \alg converges with linear rate ${\theta = \gamma/\beta\in (0,1)}$ within a radius 
    $$
    R = \rho\frac{(\sigma + \omega)}{(1-\theta)}.
    $$
\end{thm}

\begin{proof}
    The proof consists in showing that conditions ${(a)-(d)}$ of Theorem~\ref{th:mainth} hold, indeed: Paracontractivity of the operator ruling the standard dynamics holds by assumption with $\gamma\in(0,1)$;  Boundedness of the variation of the TSI holds by Lemma~\ref{lem:DOOT_TSI} with $B=\rho\sigma/\sqrt{p}$;  Boundedness of the arrival process holds by Lemma~\ref{lem:DOOT_BAR} with $H=\rho\sigma/\sqrt{p}$;  Boundedness of the departure process holds by assumption with $\beta \in(\gamma,1)$.
    The thesis follows by substituting $H$ and $B$ into \eqref{eq:radius}.
\end{proof}

\begin{thm}\label{th:mainth_omas_opt}
    In the scenario of Theorem~\ref{th:mainth_omas}, the open sequence of agents' estimates ${\{y_k:k\in\nat \}}$ converges linearly to the consensus set on the optimal solutions denoted by $\mathcal{C}^\star_k := \{\bone_{n_k}\otimes y^\star_{k} \ \mid \ y^\star_{k}\in\Y^\star_k\}$,  within a bound $\Delta_k$,
    $$
     \limsup_{k\rightarrow \infty} \ \frac{d(y_k,\C^\star_k)}{\sqrt{pn_k}}  = \frac{R}{\rho}\sqrt{n_k}:=\Delta_k.
    $$
\end{thm}
\begin{proof}
    See Section C in the Appendix.
\end{proof}

\begin{cor}\label{cor:mainth_omas}
Consider the scenario of Theorem~\ref{th:mainth_omas} and the following simplified cases:
\begin{itemize}
    \item[$(a)$] The network size is limited from above;
    \item[$(b)$] The network is time-invariant and not open;
    \item[$(c)$] The local costs are time-invariant.
\end{itemize}
Then, the result of Theorem~\ref{th:mainth_omas_opt} become:
\begin{itemize}
    \item $(a)$ implies that the tracking error can be upper bounded by $\Delta=\frac{\sqrt{n}(\sigma+\omega)}{(1-\gamma)}$ where $n_k\leq n$;
    \item $(b)$ implies that the convergence rate is equal to the paracontractivity constant $\theta=\gamma$ because $\beta=1$, which implies that the tracking error reduces to $\Delta=\frac{\sqrt{n}\sigma}{(1-\gamma)}$ because $\omega=0$, where $n_k= n$;
    \item $(b)\land(c)$ implies that the network achieves exact consensus on the optimal solution for any initial condition, i.e. the tracking error becomes null because $\sigma=0$.
\end{itemize}
\end{cor}

\subsection{Discussion on the set-up and proposed algorithm}
Let us discuss the set-up and Algorithm~\ref{alg:doot} from the perspective of learning. In this context, each agent has access to data sampled from a local distribution $\mathcal{D}_i$, which are then used to define the local cost function $f^i_k$ as
$$
    f_k^i(y) = \frac{1}{m_k^i} \sum_{h = 1}^{m_k^i} \ell(y; d_{k,h}^i)
$$
where $\{ d_{k,h}^i \}_{h = 1}^{m_k^i}$ are the local data, and $\ell$ is a loss function.
In principle, each agent could compute its model on the local data only, i.e., $y^{i,\star}_k \in\Y^{i,\star}_k$ (which always exists according to Assumptions \ref{as:all}(i)-(ii)). However, this model in general can have \textit{poor accuracy and generalization}. The poor accuracy is due to the limited amount of data that an agent can collect and store, while the poor generalization is due to the skewed/biased perspective that the local distribution $\mathcal{D}_i$ has of the phenomenon being analyzed.
Thus, each agent has an incentive to participate in the cooperative learning process to train a model that is more accurate (all data of all agents are involved) and more general (all distributions together offer a better perspective on the phenomenon).
This interpretation then motivates our choice to initialize an agent's variables with the local minimizer $y^{i,\star}_k $. That is, we assume that an agent has pre-trained its local model as best as possible with the given data, and then joins the network to cooperatively refine this model into a more accurate and general one.

The agents have therefore an incentive to join the cooperative learning to improve their local model -- but practical constraints may limit their participation. For instance, in the case of networks of battery-powered devices -- such as smartphones -- then, depending on the battery level, an agent may choose to stop participating to preserve battery, and join again only once its charge is restored. An agent could also decide to leave the cooperative learning once the trained model exceeds a high enough test accuracy on its local distribution.
However, in principle the joining/leaving of agents might disrupt the cooperative learning, resulting in trained models with poor accuracy. Thus, some assumptions on the time-variability of~\eqref{eq:optprob} are needed, in order to delineate a solvable open learning problem.

First of all, through Assumption \ref{as:all}(iii) we impose a bound to the variation of the global solutions over time. This implies that changes in the participating agents and their local dataset do not significantly disrupt the optimal trained model. As a consequence, the model trained until time $k-1$ is a good starting point for training the new model at time $k$.
Secondly, we need to guarantee that suitable bounds on the departure and arrival process are satisfied.
Assumption \ref{as:all} ensures that Definition~\ref{def:bounded-arrival} holds by requiring that local models (which are the initializations of arriving agents) are not too far from the current best global model. This translates in asking that the local distributions should have, in a sense, a bounded distance. This still allows for heterogeneity of the local data, but limits the level of the heterogeneity in the local models.
In a similar way, the leaving of an agent may drive away the current global solution to the local ones, thus requiring to assume that there is a bounded rate of departure in Theorem \ref{th:mainth_omas}. 

\begin{table*}[!ht]
\centering

\caption{Comparison with the state of the art for distributed tracking (also called dynamic consensus)\\ in open multi-agent systems with discrete-time dynamics.}
\label{tab:comparison_consensus}
    \begin{tabular}{lcccccccc}
    \hline
    [Ref.] & Problem  & \thead{Assumptions\\on the network} & \thead{Assumptions\\on the signals} & \thead{Bounded \\ tracking error} & \thead{Null \\steady-state error} \\
    \hline
    
    \cite{Deplano24stability} Deplano \textit{et al.} (2024) & Max/Min  & \thead{Undirected + Connected +\\Bounded diameter + Slowly varying} & \thead{Bounded variation +\\Bounded span} & \cmark & \thead{\mmark\\ (arbitrarily small)} \\

    \cite{abdelrahim2017max} Abdelrahim \textit{et al.} (2017)& Max/Min & \thead{Complete +\\ Eventually closed} & Constant & \xmark & \cmark  \\

    \cite{zhu2010discrete} Zhu \textit{et al.} (2010) & Avg &  \thead{Directed + Balanced +\\ jointly strongly connected 
    } &  \thead{Relative\\ bounded variation} & \xmark & \xmark  \\

    \cite{franceschelli2020stability} Franceschelli \textit{et al.} (2020)& Avg & \thead{Undirected + Connected +  \\ Bounded departures} &  \thead{Bounded variation +\\Bounded span} &\cmark & \xmark  \\

    \cite{dashti2022distributed} Dashti \textit{et al.} (2022) & Avg/Mode  &  \thead{Undirected + Connected +\\ Eventually closed} & Constant & \xmark & \cmark  \\

    \cite{makridis2024average} Makridis \textit{et al.} (2024) & Avg &  \thead{Directed + Strongly connected +\\ Eventually closed} & Constant & \xmark & \cmark  \\

    \cite{oliva2023sum} Oliva \textit{et al.} (2023) & Avg &  \thead{Undirected + Connected +\\ Eventually closed} & Constant & \xmark & \cmark  \\

    \hline

    \hline

    [\textbf{This work}] & \thead{Avg, Max/Min, \\Median, $\cdots$} &  \thead{Undirected + Connected + \\ Bounded departures} & \thead{Bounded variation +\\Bounded span}  & \cmark & \cmark  \\
    \hline
    \end{tabular}

\end{table*}

\section{Applications and numerical simulations}\label{sec:simu}
In this section, we demonstrate the performance of the \alg algorithm proposed in Section \ref{sec:algorithm} by showcasing its application to tracking and classification problems in open multi-agent systems.

\subsection{Consensus algorithms for open networks}\label{sec:simu_dc}
Consider an OMAS in which the $i$-th agent with $i\in\V_k$ has access to a scalar, time-varying reference signal ${u^i_k\in\mathbb{R}}$. 
The dynamic consensus problems on the average, the maximum, and the median values of these signals can be recast as a time-varying optimization problem in~\eqref{eq:optprob} with local cost functions -- satisying Assumption~\ref{as:all}(i)-(ii) -- given by
\begin{equation}\label{eq:dyn_cons}
f^i_k(y):= \frac{1}{q}\abs{y-u^i_k}^q + \iota_{\S^i_k}(y)
\end{equation}
where $\iota_{\S^i_k}:\rea^n\rightarrow\rea^n\cup\{+\infty\}$ is the indicator function defined as $\iota_{\S^i_k}(y) = 0$ if $y \in \S^i_k$, and $\iota_{\S^i_k}(y)=+\infty$ otherwise.
Indeed, we have the following result.
\begin{prop}\cite[Proposition 1]{deplano2023unified}\label{prop:diffprob}
Consider an open network ${\G_k=(\V_k,\E_k)}$ distributedly solving the optimization problem in~\eqref{eq:optprob} with local costs in~\eqref{eq:dyn_cons} under Assumption \ref{as:graph}. Then, there is a unique solution $y^*_k\in\Y^\star_k$ to the problem such that:
\begin{itemize}
    \item[i)] If $q=2$ and $\S^i_k=\rea$, then $y^\star_k=\avg({u_k})$;
    \item[ii)] If $q=2$ and $\S^i_k=\{x\geq u_i(k)\}$, then $y^\star_k=\max({u_k})$;
    \item[iii)] If $q=1$ and $\S^i_k=\rea$, then $y^\star_k=\med({u_k})$.
\end{itemize}
\end{prop}
The updates of the \alg for the tracking (or alternatively, dynamic consensus) problems of the average, maximum, and median values can be written in closed-form. In particular, the initialization of the new state variables in eqs.~\eqref{eq:init-x} and~\eqref{eq:init-x-2} becomes ${x^{ij}_k = \rho u^i_k}$ because the local optimal solution is unique and equal to $y^{i,\star}_k = u^i_k$; the updates of the state variables in~\eqref{eq:admm-x} remain unchanged;  the updates of the output variables in~\eqref{eq:admm-y} becomes (cfr.~\cite[Lemmas 1-2-3]{deplano2023unified}):
\begingroup
\medmuskip=1mu
\thinmuskip=1mu
\thickmuskip=1mu
$$
\begin{aligned}
    \text{(average)} && y^i_k =& \displaystyle \frac{u^i_k + \sum_{j \in \N^i_k} x^{ij}_{k}}{1 + \rho \eta^i_k},\\
    \text{(maximum)} && y^i_k =& \max\left\{u^i_k,\frac{u^i_k + \sum_{j \in \N^i_k} x^{ij}_{k}}{1 + \rho \eta^i_k}\right\},\\
    \text{(median)} && y^i_k =& u^i_k + \max\{\theta^{i,-}_k -u^i_k,0\} + \min\{\theta^{i,+}_k -u^i_k,0\},\\
    && &\text{with } \theta^{i,\pm}_k = \frac{1}{\rho \eta^i_k}\Big[\sum_{j\in\N^i_k}x^{ij}_{k-1}\pm 1\Big].
\end{aligned}
$$
\endgroup
It can also be verified that the general conditions required by Assumption~\ref{as:all}(iii)-(iv) hold if the reference signals $u^i_k\in\rea$ with $i\in\V_k$ are such that~\cite{Deplano21,deplano2023unified}:
\begin{itemize}
 \item their absolute variation is bounded by a constant $\sigma\geq 0$,
 \begin{equation}\label{eq:bounded_remu}
 \abs{u^i_{k}-u^i_{k-1}} \leq \sigma, \qquad \forall k\geq 0.
 \end{equation}
 \item they lie within a set of size $\omega\geq 0$,
 \begin{equation}\label{eq:bounded_openu}
 \abs{\bar{u}_{k}-\ubar{u}_k} \leq \omega,\qquad \forall k\geq 0.
 \end{equation}
\end{itemize}

\vspace{-1em}

\begin{rem}[Comparison with the state-of-the-art]\label{rem:sota_dc}
{A~comparison of the working assumptions of the proposed protocols derived from \alg and their performance with the state-of-the-art is detailed in Table \ref{tab:comparison_consensus}.
The only algorithms accounting for directed communications are those provided in \cite{zhu2010discrete} and in \cite{makridis2024average}, but their tracking error is not formally characterized. This is the most common case as the network is usually assumed to be eventually closed and the algorithm is characterized only at steady state. In contrast, the proposed \alg and the algorithms proposed in \cite{Deplano24stability}, \cite{franceschelli2020stability} work under the stronger assumption of undirected communications, but enjoy an eventually bounded tracking error. Moreover, the proposed \alg is the only algorithm guaranteeing a null steady state error while accounting for time-varying reference signals, whereas that in \cite{Deplano24stability} can be made arbitrarily small.
} \vspace{-0.5em}
\end{rem}

Figure~\ref{fig:DOOT_avg_res} shows a typical realization of a network of agents running the proposed \alg algorithm and compare it with the OPDC algorithm~\cite{franceschelli2020stability} in the scenario described next.
We considered as tuning parameters $\rho = 0.5$, $\alpha=0.99$ for \alg and $\alpha=\varepsilon=0.01$ for OPDC. 
The network starts with $200$ agents whose state is initialized uniformly at random in the interval $[0,500]$.
The initial graph is randomly generated with edge probability $p = 0.1$.
At any subsequent step $k\geq 0$, there is a probability $p\join_{k}\in[0,1]$ that one node joins the network and there is a probability ${p\leave_{k}\in[0,1]}$ that one node leaves the network, selected as follows: 
\begingroup
$$
[p\join_k,p\leave_k]=    
\begin{cases}
[1\%,\:1\%] & \text{if } k \leq 1000\\
[10\%,\:1\%] & \text{if } k\in (1000, 2000]\\
[1\%,\:1\%] & \text{if } k\in (2000,3000]\\
[1\%,\:10\%] & \text{if } k\in (3000, 3500]\\
[5\%,\:5\%] & \text{if } k>3500 
\end{cases}.
$$
\endgroup
We model these events in a way that the network remains connected whichever event occurs.
\begin{figure}[!b]
    \centering
    \vspace{-2em}
    \includegraphics[width=0.98\linewidth]{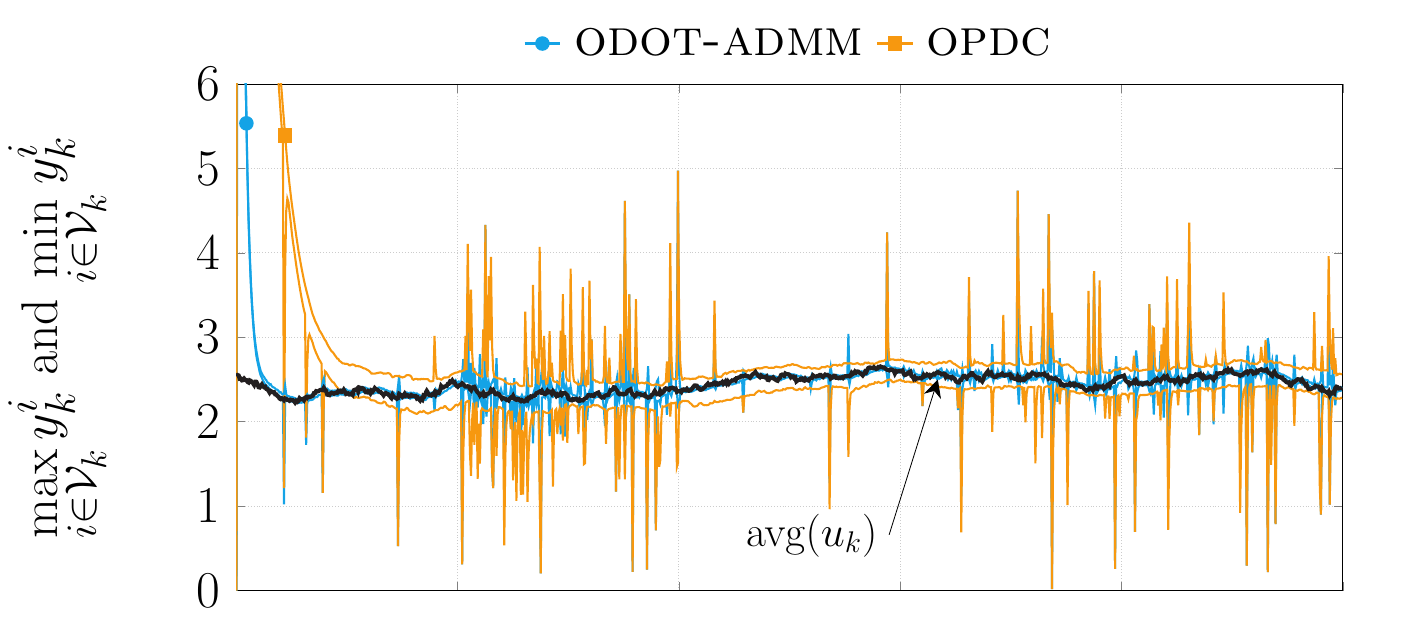}
    \includegraphics[width=0.95\linewidth]{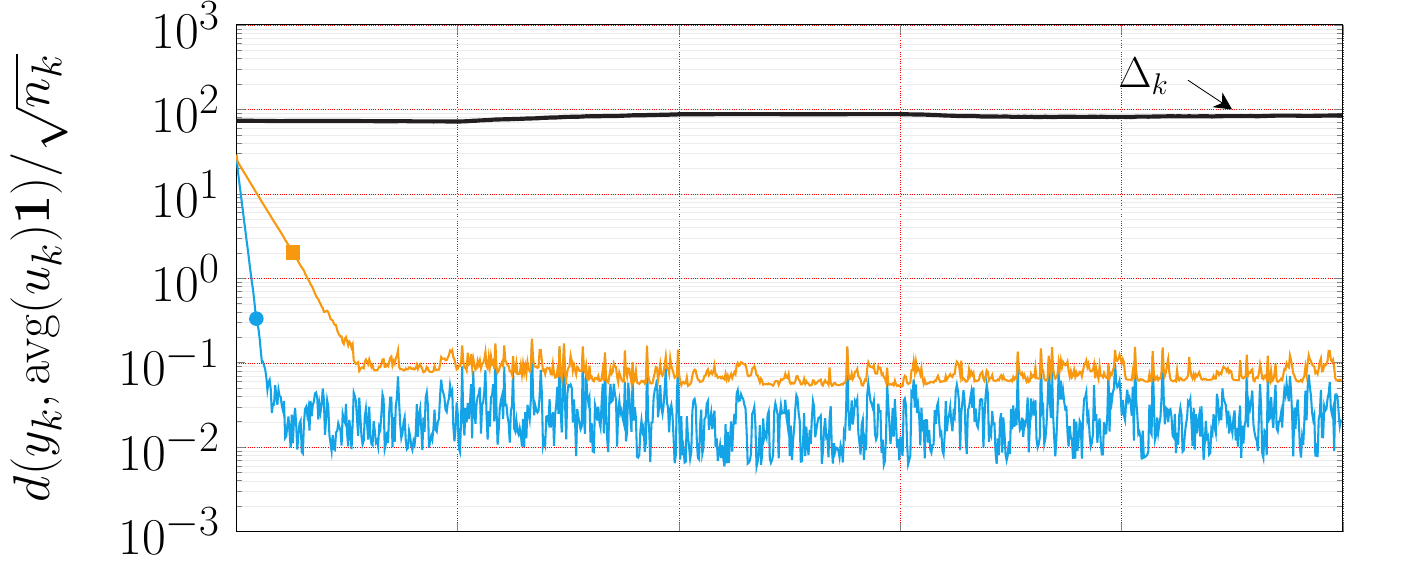}
    \includegraphics[width=0.98\linewidth]{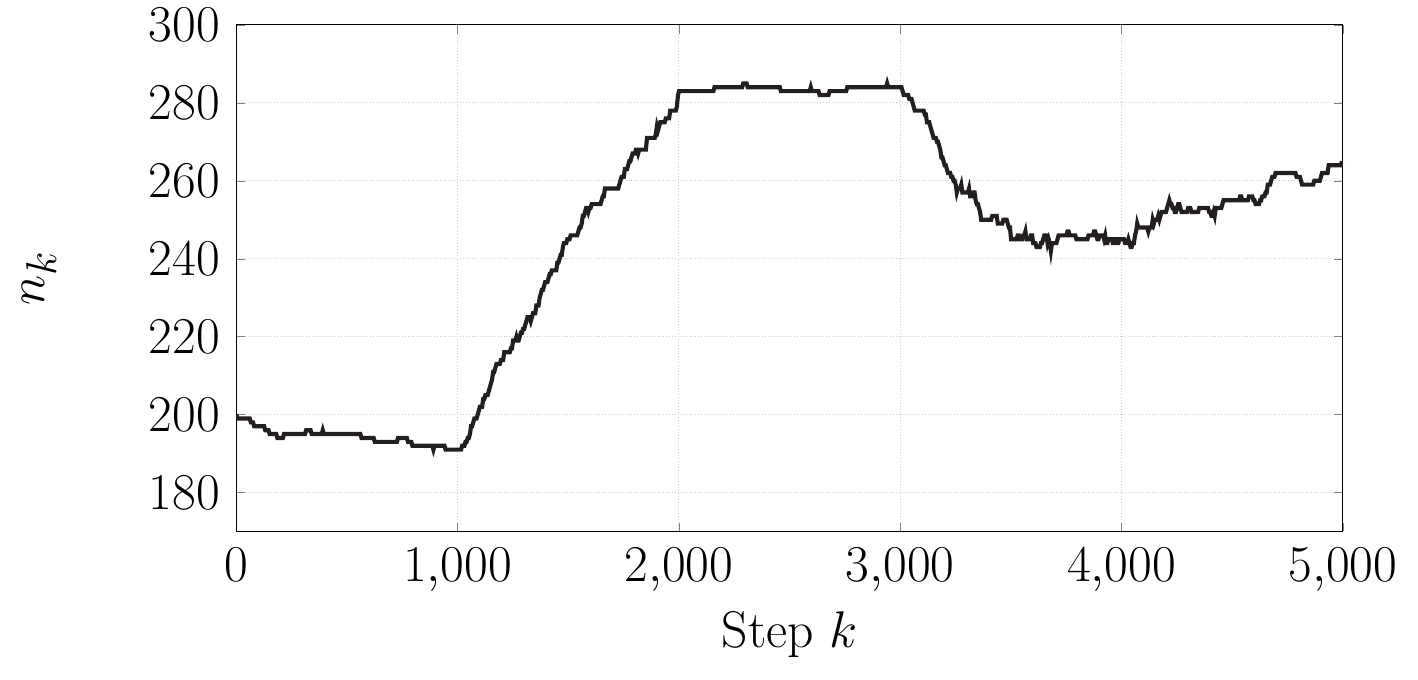}\vspace{-0.5em}
    \caption{Comparison between \alg and OPDC~\cite{franceschelli2020stability} in an open network: (top) evolution of the network state estimation $y_k$ and (middle) its normalized distance from the consensus on the average $\avg(u_k)$ of the reference signals -- converging within a bound $\Delta_k$ according to Theorem~\ref{th:mainth_omas} -- during the (bottom) time-varing size of the network.}
    \label{fig:DOOT_avg_res}
\end{figure}
Consequently, the set of network agents is frequently renewed and the number of agents changes according to Fig.~\ref{fig:DOOT_avg_res}(bottom).

The arriving agent creates random communication channels with all other agents with probability $p=0.1$.
Input reference signals are randomly sampled in the interval $[0, 5]$ when agents join the network and evolve within this set with bounded variation given by $\sigma=0.2$, yielding $\omega=5$.
Figure~\ref{fig:DOOT_avg_res}(top) shows the highest and the lowest agents' state estimations $y^i_k$ and the objective signal to be trakced, that is the average $\avg(u_k)$ of the reference signals. Figure~\ref{fig:DOOT_avg_res}(middle) shows the normalized distance of the whole network state $y_k= [y^1_k,\cdots,y^{n_k}_k]^\top$ from the consensus on the average. One can verify that the convergence of both protocols is linear and approaches a bounded tracking error. In this simulation, \alg outperforms OPDC both for the faster convergence rate and the smaller tracking error, which is bounded coherently with Theorem~\ref{th:mainth_omas_opt}.
In particular, we considered the best possible convergence radius $R=\rho(\sigma+\omega)=0.52$ by letting $\gamma=0$, thus yielding the bound $\Delta_k=5.2\sqrt{n_k}$. It is evident that this bound on the tracking error is a quite conservative estimate of the actual error.

\begin{figure}[!b]
    \centering
    \includegraphics[width=0.95\linewidth]{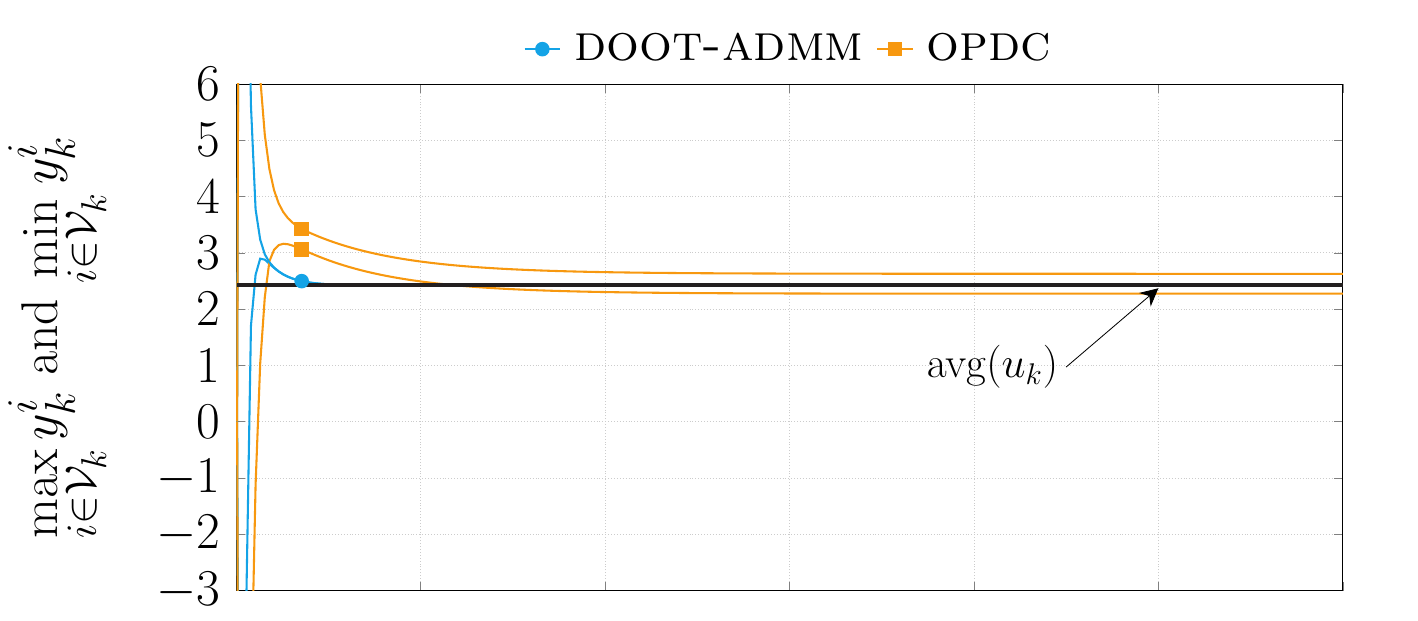}
    \includegraphics[width=0.98\linewidth]{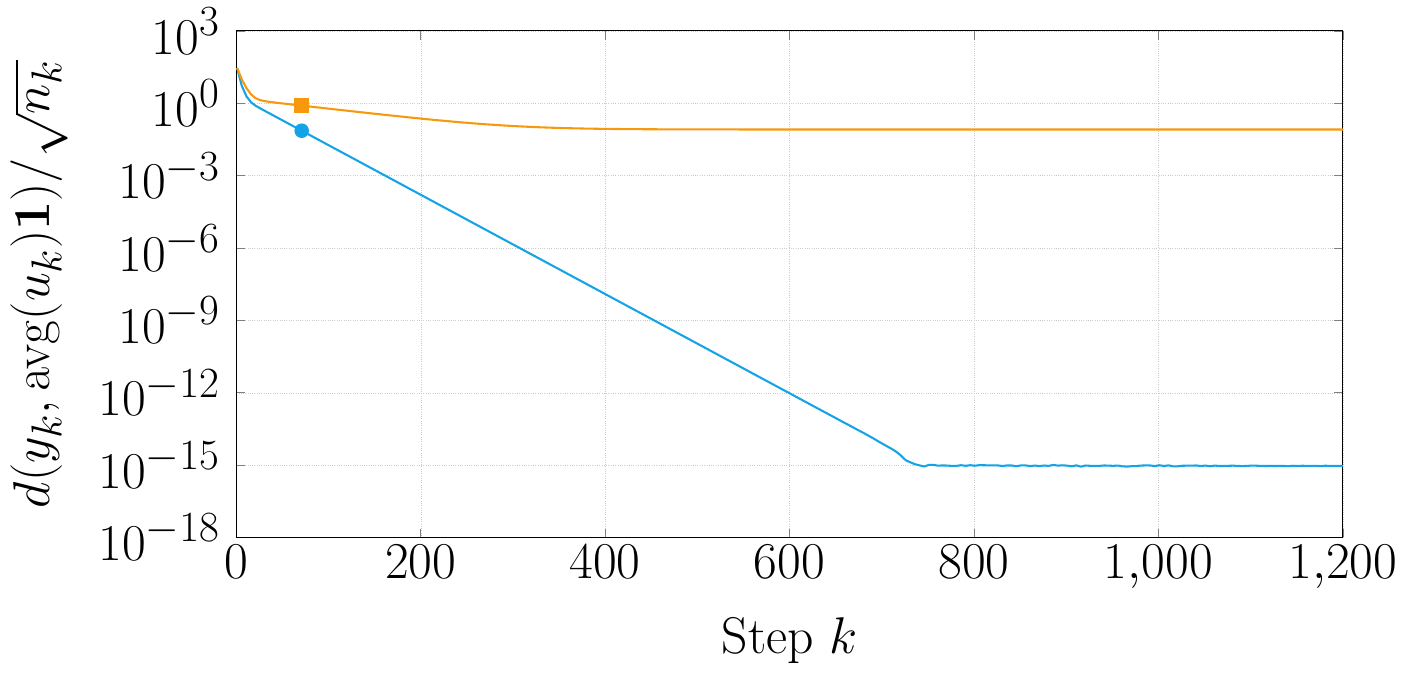}
    \caption{Comparison between DOT-ADMM and PDC in~\cite{franceschelli2020stability} in a closed network: (top) evolution of the network state estimation $y_k$ and (middle) its normalized distance from the consensus on the average $\avg(u_k)$ of the reference signals -- converging to zero according to Corollary~\ref{cor:mainth_omas}.}
    \label{fig:DOOT_avg_res_c}
\end{figure}

Figure~\ref{fig:DOOT_avg_res_c} shows the behavior of the two algorithms when the network is closed and fixed with $n=200$ agents and the reference signals are static: in this case, the two algorithms are called DOT-ADMM~\cite{bastianello2024robust} and PDC~\cite{franceschelli2020stability}. It can be noticed that DOT-ADMM is able to converge to the desired average value of the reference signal -- up to machine precision -- according to Corollary~\ref{cor:mainth_omas}, while PDC converges with a bounded error.

\subsection{Learning over open networks}

In this section we apply the \alg to a classification problem, characterized by the (static) local costs
$$
    f^i_k:= f_i(x) = \frac{1}{m_i} \sum_{j = 1}^{m_i} \log\left( 1 + \exp\left( - b_{i,j} a_{i,j}^\top x \right) \right) + \frac{\epsilon}{2} \norm{x}^2
$$
where $a_{i,j} \in \rea^n$ and $b_{i,j} \in \{ -1, 1 \}$ are the pairs of feature vector and label, randomly generated with \texttt{scikit-learn}, and $\epsilon = 0.05$.
In the following sections we test \alg on several open networks, with different models of arriving/remaining/departing events. We also test some variants of \alg which differ in how arriving agents initialize their local states.

\begin{figure}[!b]
    \centering
    \includegraphics[width=0.95\linewidth]{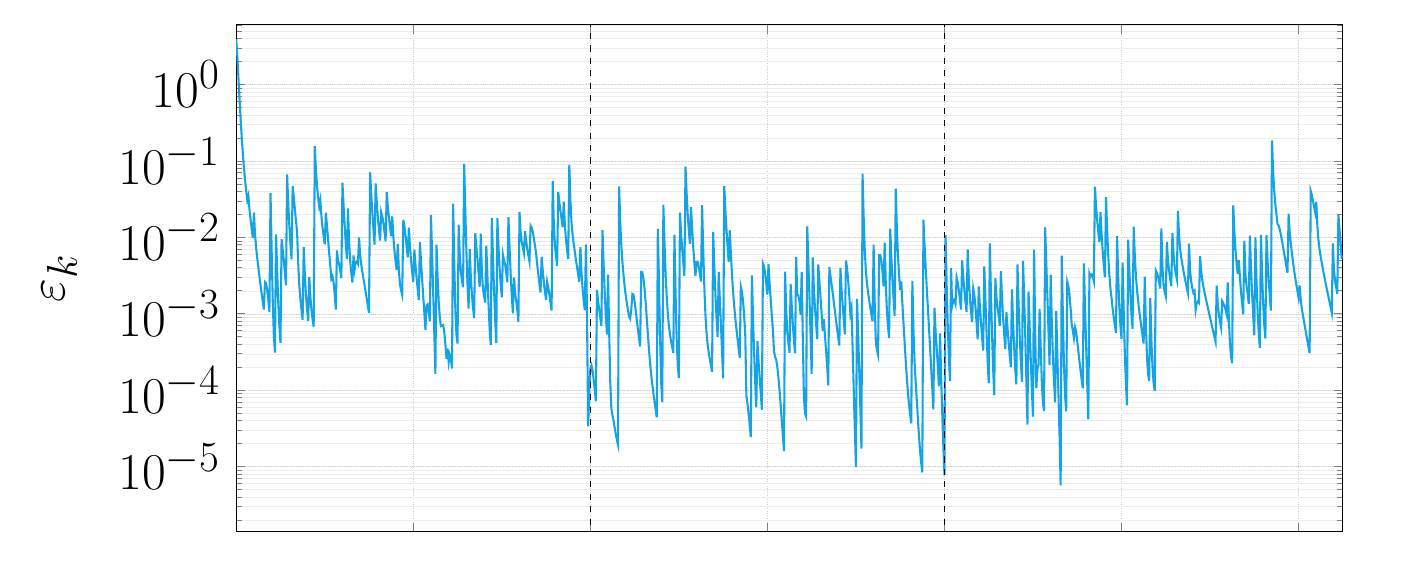}
    \includegraphics[width=0.98\linewidth]{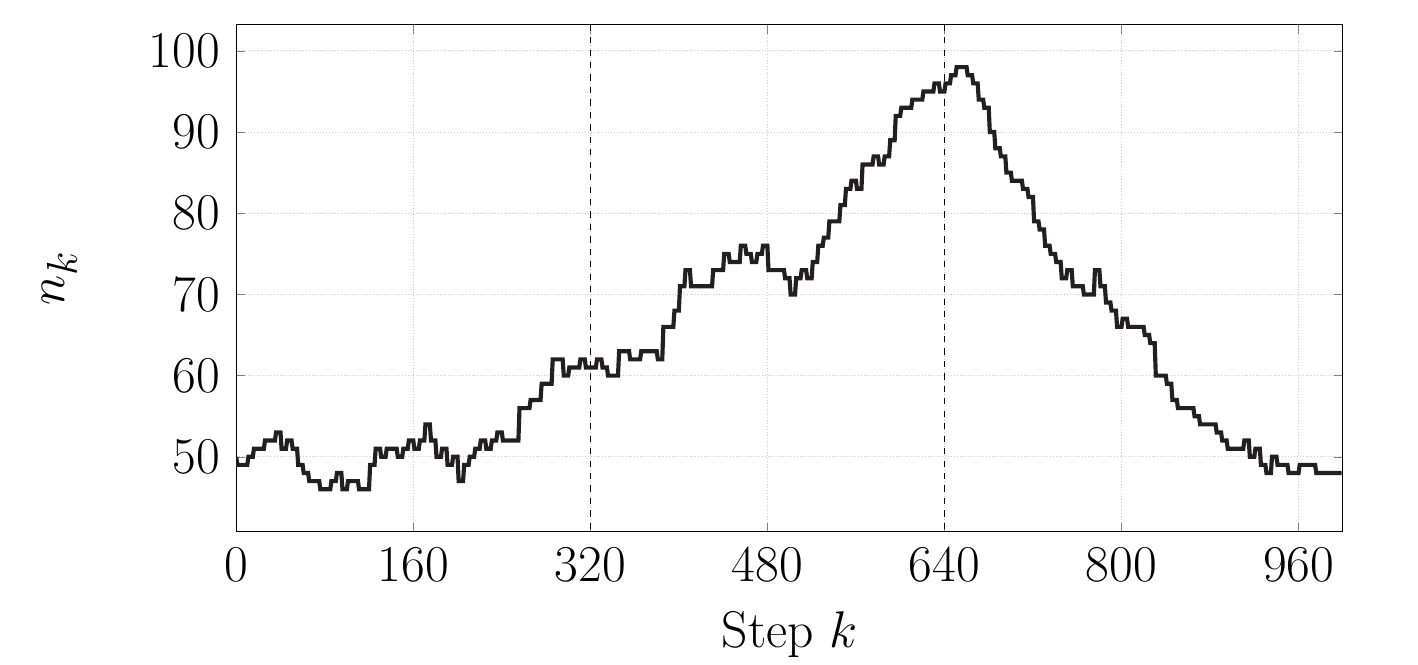}
    
    \caption{\alg applied to an open network with three modes: $\lambda\join = \lambda\leave$, $\lambda\join = 2 \lambda\leave$, $\lambda\join = 0.5 \lambda\leave$.}
    \label{fig:changing_mode}
\end{figure}

We remark that the local costs do not have a closed form $\prox$ for the selected logistic loss; rather, the agents approximate it with accelerated gradient descent, up to precision of $10^{-10}$. 
For the same reason, it becomes quite impractical to numerically evaluate the distance of the agents estimations from the global optimal solution.
Thus, in all following simulations, we use instead the following metric as a proxy of the algorithm's convergence to the solution of~\eqref{eq:optprob}:
\begin{equation}\label{eq:proxymetric}
    \varepsilon_k := \norm{\sum_{i \in\V_k} \nabla f_i(\bar{y}_k)}^2,\quad\text{with} \quad \bar{y}_k = \frac{1}{n_k} \sum_{i \in\V_k} y_k^i
\end{equation}
The above metric is zero if all the $y_k^i$, $i \in \V_k$, have converged to the optimal solution.

\subsubsection{Open network with Poisson arrivals/departures}
We start by applying \alg to an open network in which the arrival and departure events occur according to the Poisson distributions $\text{Pois}(\lambda\join)$ and $\text{Pois}(\lambda\leave)$, respectively -- resulting in $\lambda\join$ and $\lambda\leave$ arrivals and departures in mean. Additionally, arriving agents are connected to a number of remaining agents equal to the average degree in the network, and the network starts with $n_0 = 50$ agents.
In Figure~\ref{fig:changing_mode} we report the evolution of $\varepsilon_k$ and the number of agents $n_k$ over the course of a simulation where the network is characterized by three modes:
\begingroup
$$
[\lambda\join_k, \lambda\leave_k]=    
\begin{cases}
[1,\: 1] & \text{if } k \leq 320\\
[1,\: 0.5] & \text{if } k\in (320, 640]\\
[0.5,\: 1] & \text{if } k> 640]
\end{cases}.
$$
\endgroup
We can see that, after the initial transient, $\varepsilon_k$ remains upper bounded, as predicted by our theoretical results. This is the case also when $\lambda\join = 2 \lambda\leave$ and $\lambda\join = 0.5 \lambda\leave$, which result in an increasing and decreasing (in mean) $n_k$, respectively. This observation is especially important, as it showcases the resilience of \alg to wide changes in the network.

The values of $\lambda\join$ and $\lambda\leave$ of course affect that magnitude of the fluctuations in $n_k$ that happen over the course of the simulation. To explore how this in turn impacts $\varepsilon_k$, we run a set of simulations with $\lambda\join = \lambda\leave$ and different choices of $\lambda\join \in \{ 0.1, 1, 10, 100 \}$. The results are reported in Figure~\ref{fig:changing_lambdas}.
We notice that the larger the mean of the arrival/departure events, the larger $\varepsilon_k$ is, due to the wider fluctuations in the number of agents.
This is further verified by the results of Table~\ref{tab:poisson_mean}, which report the minimum, maximum, mean and standard deviation of $\varepsilon_k$ in the second half of the simulation (to exclude transient behaviors due to the initialization). The results are averaged over $10$ Monte Carlo iterations.
The results confirm that the more arrival/departure events, the larger $\varepsilon_k$, in line with our theoretical results. We remark that the algorithm does not diverge even in the challenging scenario $\lambda = 100$.

\begin{figure}[!b]
    \centering
    \includegraphics[width=0.95\linewidth]{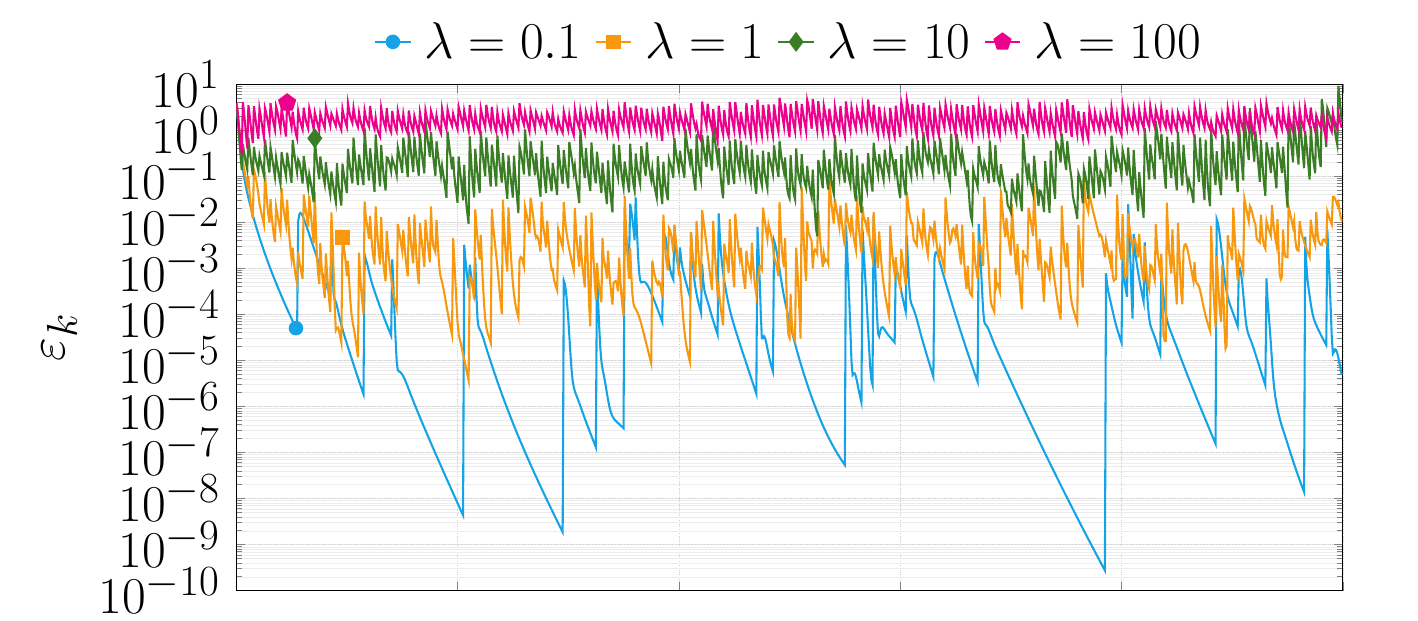}

    \includegraphics[width=0.98\linewidth]{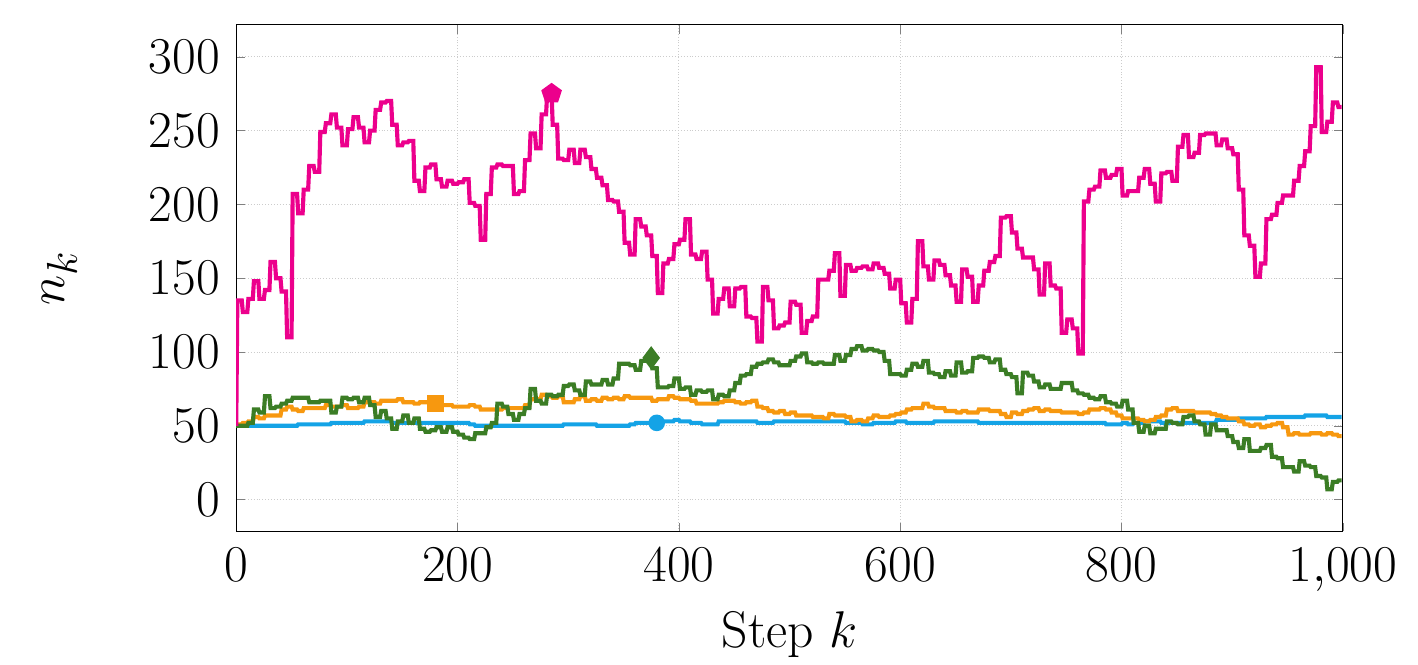}
    
    \caption{\alg applied to an open network with $\lambda\join = \lambda\leave=\lambda \in \{ 0.1, 1, 10, 100 \}$.}
    \label{fig:changing_lambdas}
\end{figure}

\subsubsection{Open network with decaying Poisson}
The previous section tested \alg in scenarios where arrival and departure events keep happening with a constant probability. In this section, we test \alg in a scenario where the arrival and departure events become rarer as time goes on, that is, they have distributions $\text{Pois}(\lambda \delta^{k/5})$ and $\text{Pois}(\lambda \delta^{k/5})$, with $\lambda = 5$ and $\delta = 0.9583$. \vspace{-0.5em}
\begin{table}[!h]
    \centering
    \caption{Performance of \alg in terms of $\varepsilon_k$ in \eqref{eq:proxymetric} for different network sizes, and only replacement events.}
    \begin{tabular}{c|ccc}
    $\lambda$ & Min & Mean $\pm$ Std & Max \\
    \hline
    $0.1$ & $6.163 \times 10^{-5}$ & $( 6.765 \pm 7.988) \times 10^{-4}$ & $4.192 \times 10^{-3}$ \\
    $1$ & $3.178 \times 10^{-3}$ & $( 1.701 \pm 1.012) \times 10^{-2}$ & $5.655 \times 10^{-2}$ \\
    $10$ & $6.550 \times 10^{-2}$ & $( 1.765 \pm 0.968) \times 10^{-1}$ & $5.698 \times 10^{-1}$ \\
    $100$ & $9.282 \times 10^{-1}$ & $1.692 \pm 0.624$ & $3.046$ \\
    \end{tabular}
    \label{tab:poisson_mean}\vspace{-1em}
\end{table}
This models a scenario where at the initial learning stage there is high churn rate, but progressively the agents decide whether to participate or not, and the network settles.

In Figure~\ref{fig:decaying} we report the evolution of $\varepsilon_k$ and of the number of agents in the network.
We notice that, similarly to the results of the previous section, while the arrival/departure events are frequent, \alg achieves a bounded error. But as the probabilities decay, the network settles on an increasingly static structure, and the algorithm achieves exact convergence around iteration $800$ (up to numerical precision). The probabilities of arrival/departure are not zero though, and we see that an agent enters the network at $\sim 850$, causing a transient in the error, which then converges to zero again.

\begin{figure}[!b]
    \centering
    \includegraphics[width=0.98\linewidth]{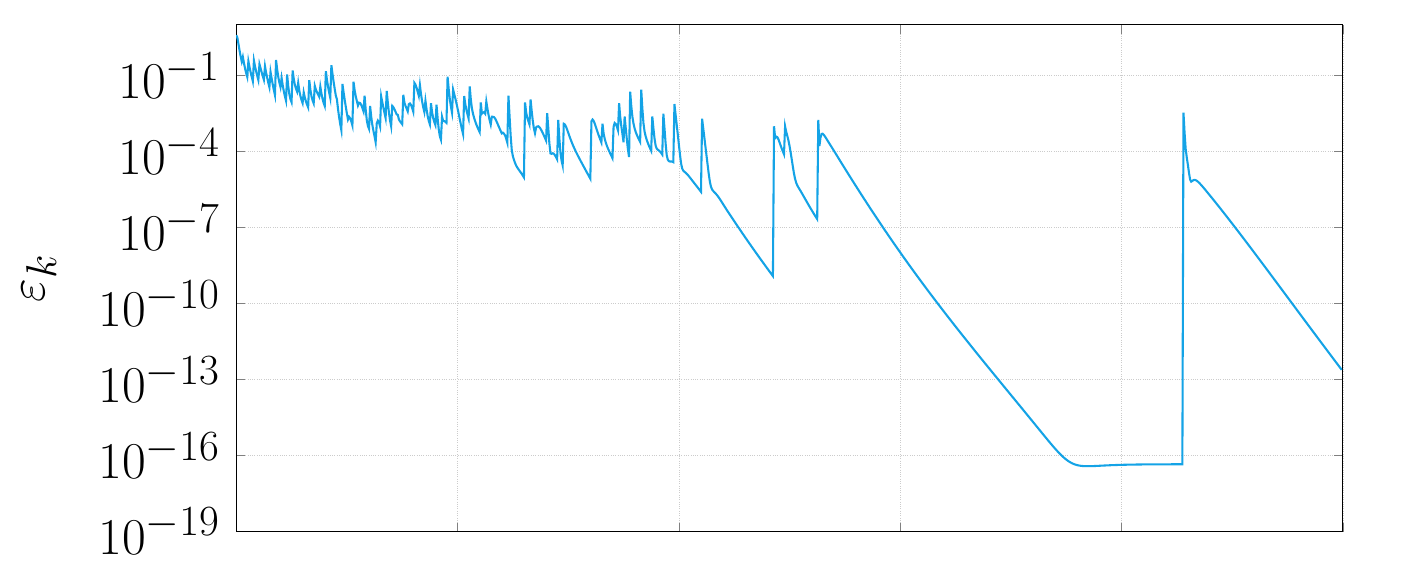}

    \includegraphics[width=0.95\linewidth]{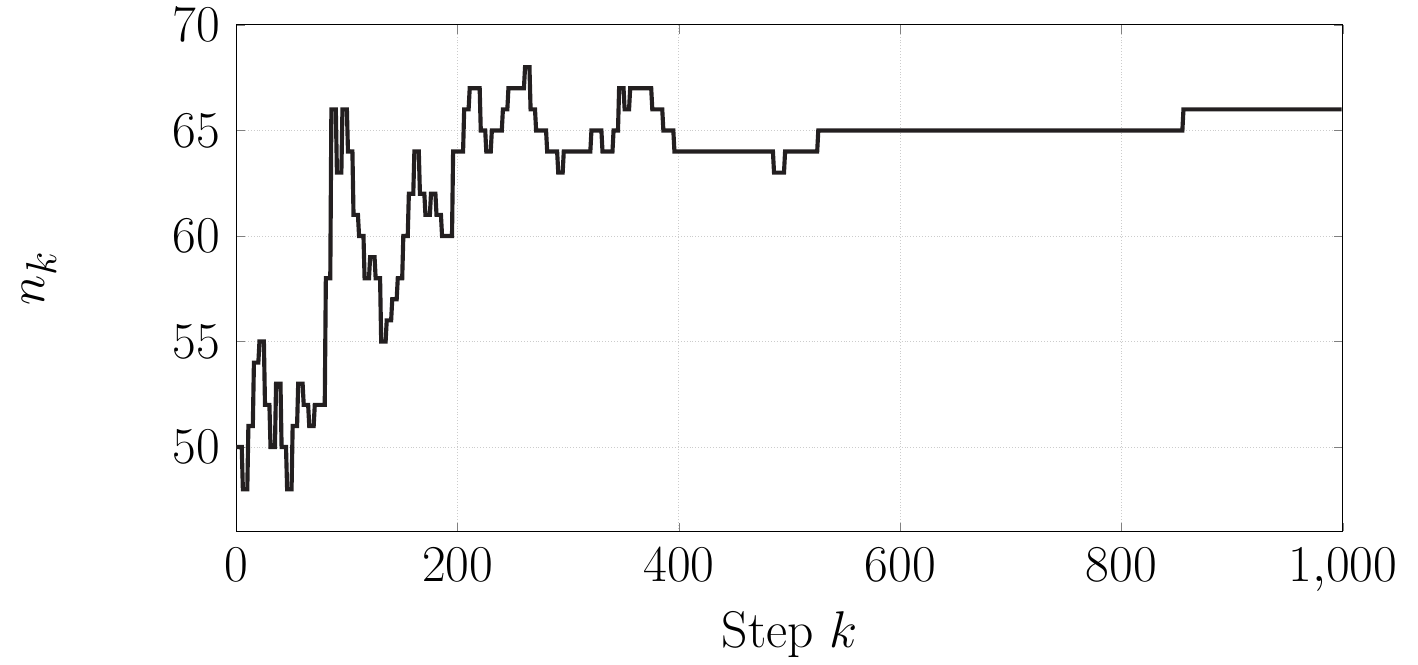}
    
    \caption{\alg applied to an open network with decaying arrival/departure events probability.}
    \label{fig:decaying}
\end{figure}

\subsubsection{Open network with replacement}
The open network models used in the previous sections allowed for the number of agents $n_k$ to vary over time. In this section we consider a network with a fixed number of agents $n$, but allow a number of them -- drawn from $\text{Pois}(\lambda)$, $\lambda = 1$ -- to be replaced throughout the simulation.

Table~\ref{tab:replacement} reports the minimum, maximum, mean and standard deviation of $\varepsilon_k$ in the second half of the simulation, with $n \in \{ 50, 100, 500 \}$. The results are averaged over $10$ Monte Carlo iterations.
We notice that the higher the number of agents, the smaller the value of $\varepsilon_k$. Indeed, in larger networks the impact of replacing a few nodes ($1$ on average) is lesser, as the large number of cost functions in~\eqref{eq:optprob} ensure smaller sensitivity to individual changes in the agents.
However, $\varepsilon_k$ does not converge to zero, as changes in the network still cause transients similar to that at $\sim 850$ of Figure~\ref{fig:decaying}.

\begin{table}[!h]
    \centering
    \caption{Performance of \alg in terms of $\varepsilon_k$ in \eqref{eq:proxymetric} for different network sizes, and only replacement events.}
    \begin{tabular}{c|ccc}
    $n$ & Min & Mean $\pm$ Std & Max \\
    \hline
    $50$ & $1.661 \times 10^{-3}$ & $( 1.253 \pm 0.980) \times 10^{-2}$ & $6.159 \times 10^{-2}$ \\
    $100$ & $3.997 \times 10^{-4}$ & $( 3.685 \pm 4.535) \times 10^{-3}$ & $3.382 \times 10^{-2}$ \\
    $500$ & $7.493 \times 10^{-5}$ & $( 4.676 \pm 3.045) \times 10^{-4}$ & $1.817 \times 10^{-3}$ \\
    \end{tabular}
    \label{tab:replacement}\vspace{-1.5em}
\end{table}

\subsubsection{Initialization in \alg}
The agents running \alg use their local optimum as an initialization when they join the network (or when they connect to a new neighbor). However, in practice the agents may not have access to a locally optimal model (especially if their cost function changes over time). Thus, in this section we test the performance of \alg with different local initializations: (a) the local optimum $x_k^{ij} = \rho y^{i,\star}$; (b) the zero vector $x_k^{ij} = 0$; (c) the average state of $i$'s neighbors that were already part of the network $x_k^{ij} = \avg\left( \{ y_{k-1}^j \}_{j \in \N_k^i \cup \R_k} \right)$. The choice of (c) enacts a ``knowledge transfer'' between remaining and arriving agents; this is similar to the sharing of subgradients in \cite{hsieh_optimization_2021}.

Table~\ref{tab:initialization} reports the minimum, maximum, mean and standard deviation of $\varepsilon_k$ (in the second half of the simulation) with the different initializations. The results are averaged over $10$ Monte Carlo iterations.
First of all, we remark that the na\"ive choice (b) leads to the worst performance overall, due to the fact that it causes the largest transient effect. Indeed, the zero vector is (on average) much farther from the fixed point than (a) and (c).
On the other hand, these two better informed initializations have very close performance, with (c) being slightly worse. As mentioned above, (a) may not be accessible due to the computational cost of computing the local optimum. Thus (c) presents a less computationally expensive alternative which relies on communications rather than local computation.

\begin{table}[!ht]
    \centering
    \caption{Performance of \alg in terms of $\varepsilon_k$ in \eqref{eq:proxymetric} for different network sizes, and only replacement events.}
    \begin{tabular}{c|ccc}
    Init. & Min & Mean $\pm$ Std & Max \\
    \hline
    (a) & $3.178 \times 10^{-3}$ & $( 1.701 \pm 1.012) \times 10^{-2}$ & $5.655 \times 10^{-2}$ \\
    (b) & $2.665 \times 10^{-3}$ & $(2.214 \pm 3.854) \times 10^{-1}$ & $1.692$ \\
    (c) & $3.212 \times 10^{-3}$ & $(2.134 \pm 1.630) \times 10^{-2}$ & $7.773 \times 10^{-2}$ \\
    \end{tabular}
    \label{tab:initialization}
\end{table}

\section{Conclusions}\label{sec:conclusion}
This article presents the \alg algorithm to solve distributed optimization and learning problems in networks where agents may join or leave during the execution of the algorithm.
The stability and performance of \alg are discussed in the light of the newly introduced \virg{\textit{open operator theory}}, which are corroborated through extensive simulations in the to solve the dynamic consensus problems on different metrics and classification problems through logistic regressions.
The superiority of our approach with respect to the state-of-the-art has been discussed both in terms of working assumptions and performance, as detailed in Tables \ref{tab:comparison_optimization}-\ref{tab:comparison_consensus}.

Many interesting future research directions originates from this manuscript, which mainly reside in providing new the technical tools for \virg{\textit{open operator theory}} to correctly describe and analyze more complex and realistic scenarios, such as asynchronous communications among the agents, unreliable and limited communications to deal with packet losses and quantization limitation, inexact local computations, and so on. 

\newpage

\printbibliography

@inproceedings{hsieh_optimization_2021,
	address = {Austin, TX, USA},
	title = {Optimization in {Open} {Networks} via {Dual} {Averaging}},
	booktitle = {2021 60th {IEEE} {Conference} on {Decision} and {Control} ({CDC})},
	publisher = {IEEE},
	author = {Hsieh, Yu-Guan and Iutzeler, Franck and Malick, Jerome and Mertikopoulos, Panayotis},
	month = dec,
	year = {2021},
	pages = {514--520},
}

@article{zhou2024prescribed,
  title={Prescribed-time consensus of time-varying open multi-agent systems with delays on time scales},
  author={Zhou, Boling and Park, Ju H and Yang, Yongqing and Hao, Rixu and Jiao, Yu},
  journal={Information Sciences},
  volume={680},
  pages={120957},
  year={2024},
  publisher={Elsevier}
}

@inproceedings{hendrickx_stability_2020,
	address = {Jeju Island, Korea (South)},
	title = {Stability of {Decentralized} {Gradient} {Descent} in {Open} {Multi}-{Agent} {Systems}},
	booktitle = {2020 59th {IEEE} {Conference} on {Decision} and {Control} ({CDC})},
	publisher = {IEEE},
	author = {Hendrickx, Julien M. and Rabbat, Michael G.},
	month = dec,
	year = {2020},
	pages = {4885--4890},
}

@article{hayashi_distributed_2023,
	title = {Distributed {Subgradient} {Method} in {Open} {Multiagent} {Systems}},
	volume = {68},
	number = {10},
	journal = {IEEE Transactions on Automatic Control},
	author = {Hayashi, Naoki},
	month = oct,
	year = {2023},
	pages = {6192--6199},
}

@article{sawamura_distributed_2024,
	title = {A {Distributed} {Primal}-{Dual} {Push}-{Sum} {Algorithm} on {Open} {Multiagent} {Networks}},
	journal = {IEEE Transactions on Automatic Control},
	author = {Sawamura, Riki and Hayashi, Naoki and Inuiguchi, Masahiro},
	year = {2024},
}

@inproceedings{de2021random,
  title={Random coordinate descent algorithm for open multi-agent systems with complete topology and homogeneous agents},
  author={de Galland, Charles Monnoyer and Vizuete, Renato and Hendrickx, Julien M and Frasca, Paolo and Panteley, Elena},
  booktitle={2021 60th IEEE Conference on Decision and Control (CDC)},
  pages={1701--1708},
  year={2021},
  organization={IEEE}
}

@article{Ryu16,
  title={Primer on monotone operator methods},
  author={Ryu, Ernest K and Boyd, Stephen},
  journal={Appl. Comput. Math},
  volume={15},
  number={1},
  pages={3--43},
  year={2016}
}

@book{Bauschke2017,
author="Bauschke, Heinz H.
and Combettes, Patrick L.",
title="Convex Analysis and Monotone Operator Theory in Hilbert Spaces",
bookTitle="Convex Analysis and Monotone Operator Theory in Hilbert Spaces",
year="2017",
publisher="Springer",
address = "Cham, Switzerland",
edition = {2},
}

@article{nakamoto2008bitcoin,
  title={Bitcoin: A peer-to-peer electronic cash system},
  author={Nakamoto, Satoshi},
  journal={Satoshi Nakamoto},
  year={2008}
}

@article{rainer2002opinion,
  title={Opinion dynamics and bounded confidence: models, analysis and simulation},
  author={Rainer, Hegselmann and Krause, Ulrich},
  year={2002}
}

@article{brambilla2013swarm,
  title={Swarm robotics: a review from the swarm engineering perspective},
  author={Brambilla, Manuele and Ferrante, Eliseo and Birattari, Mauro and Dorigo, Marco},
  journal={Swarm Intelligence},
  volume={7},
  pages={1--41},
  year={2013},
  publisher={Springer}
}

@article{galland_random_2024,
	title = {Random {Coordinate} {Descent} for {Resource} {Allocation} in {Open} Multiagent {Systems}},
	journal = {IEEE Transactions on Automatic Control},
	author = {Galland, Charles Monnoyer De and Vizuete, Renato and Hendrickx, Julien M. and Panteley, Elena and Frasca, Paolo},
	year = {2024},
}

@article{liu_distributed_2024,
	title = {Distributed {Online} {Resource} {Allocation} in {Open} {Networks}},
	journal = {IEEE Transactions on Automatic Control},
	author = {Liu, Yuxuan and Ye, Maojiao and Ding, Lei and Xie, Lihua and Xu, Shengyuan},
	year = {2024},
}

@article{hsieh2022multi,
  title={Multi-agent online optimization with delays: Asynchronicity, adaptivity, and optimism},
  author={Hsieh, Yu-Guan and Iutzeler, Franck and Malick, Jérôme and Mertikopoulos, Panayotis},
  journal={Journal of Machine Learning Research},
  volume={23},
  number={78},
  pages={1--49},
  year={2022}
}

@inproceedings{vizuete2022resource,
  title={Resource allocation in open multi-agent systems: an online optimization analysis},
  author={Vizuete, Renato and de Galland, Charles Monnoyer and Hendrickx, Julien M and Frasca, Paolo and Panteley, Elena},
  booktitle={2022 IEEE 61st Conference on Decision and Control (CDC)},
  pages={5185--5191},
  year={2022},
  organization={IEEE}
}

@book{Serre12,
  title={Matrices: Theory and Applications (2nd Edition)},
  author={Serre, Denis},
  year={2012},
  publisher={ASpringer New York, NY}
}

@article{lancaster1972norms,
  title={Norms on direct sums and tensor products},
  author={Lancaster, Peter and Farahat, Hanafi K},
  journal={mathematics of computation},
  volume={26},
  number={118},
  pages={401--414},
  year={1972}
}

@article{bistritz2024gamekeeper,
  title={Gamekeeper: Online learning for admission control of networked open multiagent systems},
  author={Bistritz, Ilai and Bambos, Nicholas},
  journal={IEEE Transactions on Automatic Control},
  year={2024},
  publisher={IEEE}
}

@article{li_survey_2023,
	title = {A survey on distributed online optimization and online games},
	volume = {56},
	journal = {Annual Reviews in Control},
	author = {Li, Xiuxian and Xie, Lihua and Li, Na},
	month = jan,
	year = {2023},
	pages = {100904},
}

@article{Bastianello21,
author={Bastianello,N. and Carli,R. and Schenato,L. and Todescato,M.},
year={2021},
title={Asynchronous Distributed Optimization over Lossy Networks via Relaxed {ADMM}: Stability and Linear Convergence},
journal={IEEE Transactions on Automatic Control},
volume={66},
number={6},
pages={2620--2635},
}

@article{franceschelli2020stability,
  title={Stability of open multiagent systems and applications to dynamic consensus},
  author={Franceschelli, Mauro and Frasca, Paolo},
  journal={IEEE Transactions on Automatic Control},
  volume={66},
  number={5},
  pages={2326--2331},
  year={2020},
}

@ARTICLE{Deplano21,
  author={Deplano, Diego and Franceschelli, Mauro and Giua, Alessandro},
  journal={IEEE Transactions on Automatic Control}, 
  title={Dynamic Min and Max Consensus and Size Estimation of Anonymous Multiagent Networks}, 
  year={2023},
  volume={68},
  number={1},
  pages={202-213},
  doi={10.1109/TAC.2021.3135452}}

@inproceedings{franceschelli2018proportional,
  title={Proportional dynamic consensus in open multi-agent systems},
  author={Franceschelli, Mauro and Frasca, Paolo},
  booktitle={2018 IEEE Conference on Decision and Control (CDC)},
  pages={900--905},
  year={2018},
  organization={IEEE}
}

@inproceedings{de2019lower,
  title={Lower bound performances for average consensus in open multi-agent systems},
  author={de Galland, Charles Monnoyer and Hendrickx, Julien M},
  booktitle={2019 IEEE 58th Conference on Decision and Control (CDC)},
  pages={7429--7434},
  year={2019},
  organization={IEEE}
}

@article{zhu2010discrete,
  title={Discrete-time dynamic average consensus},
  author={Zhu, Minghui and Martinez, Sonia},
  journal={Automatica},
  volume={46},
  number={2},
  pages={322--329},
  year={2010},
  publisher={Elsevier}
}

@article{dashti2022distributed,
  title={Distributed mode computation in open multi-agent systems},
  author={Dashti, Zoreh Al Zahra Sanai and Oliva, Gabriele and Seatzu, Carla and Gasparri, Andrea and Franceschelli, Mauro},
  journal={IEEE Control Systems Letters},
  volume={6},
  pages={3481--3486},
  year={2022},
  publisher={IEEE}
}

@inproceedings{dashti2019dynamic,
  title={Dynamic consensus on the median value in open multi-agent systems},
  author={Dashti, Zohreh Al Zahra Sanai and Seatzu, Carla and Franceschelli, Mauro},
  booktitle={2019 IEEE 58th Conference on Decision and Control (CDC)},
  pages={3691--3697},
  year={2019},
  organization={IEEE}
}

@article{nakamura2023cooperative,
  title={Cooperative learning for adversarial multi-armed bandit on open multi-agent systems},
  author={Nakamura, Tomoki and Hayashi, Naoki and Inuiguchi, Masahiro},
  journal={IEEE Control Systems Letters},
  volume={7},
  pages={1712--1717},
  year={2023},
  publisher={IEEE}
}

@INPROCEEDINGS{Deplano24stability,
author={D. {Deplano} and M. {Franceschelli} and A. {Giua}},
booktitle={63rd IEEE Conference on Decision and Control (CDC)},
title={Stability of Paracontractive Open Multi-Agent Systems},
year={2024},
volume={},
number={},
pages={},}

@article{oliva2023sum,
  title={A Sum-of-States Preservation Framework for Open Multi-Agent Systems With Nonlinear Heterogeneous Coupling},
  author={Oliva, Gabriele and Franceschelli, Mauro and Gasparri, Andrea and Scala, Antonio},
  journal={IEEE Transactions on Automatic Control},
  year={2023},
  publisher={IEEE}
}

@inproceedings{makridis2024average,
    title={Average Consensus over Directed Networks in Open Multi-Agent Systems with Acknowledgement Feedback},
    author={Makridis, Evagoras and Grammenos, Andreas and Oliva, Gabriele and Kalyvianaki, Evangelia and Hadjicostis, Christoforos N and Charalambous, Themistoklis},
    booktitle={63rd IEEE Conference on Decision and Control (CDC)},
    year={2024},
    volume={},
    number={},
    pages={},
}

@inproceedings{restrepo2022consensus,
  title={Consensus of open multi-agent systems over dynamic undirected graphs with preserved connectivity and collision avoidance},
  author={Restrepo, Esteban and Lorìa, Antonio and Sarras, Ioannis and Marzat, Julien},
  booktitle={2022 IEEE 61st Conference on Decision and Control (CDC)},
  pages={4609--4614},
  year={2022},
  organization={IEEE}
}

@ARTICLE{varma2018open,
  author={Varma, Vineeth S. and Morărescu, Irinel-Constantin and Ne{\v s}ić, Dragan},
  journal={IEEE Control Systems Letters}, 
  title={Open Multi-Agent Systems With Discrete States and Stochastic Interactions}, 
  year={2018},
  volume={2},
  number={3},
  pages={375-380},
  doi={10.1109/LCSYS.2018.2840431}}

@inproceedings{hendrickx2017openrand,
  title={Open multi-agent systems: Gossiping with random arrivals and departures},
  author={Hendrickx, Julien M and Martin, Samuel},
  booktitle={2017 IEEE 56th Annual Conference on Decision and Control (CDC)},
  pages={763--768},
  year={2017},
  organization={IEEE}
}

@inproceedings{hendrickx2016opendet,
  title={Open multi-agent systems: Gossiping with deterministic arrivals and departures},
  author={Hendrickx, Julien M and Martin, Samuel},
  booktitle={2016 54th Annual Allerton Conference on Communication, Control, and Computing (Allerton)},
  pages={1094--1101},
  year={2016},
  organization={IEEE}
}

@article{xue2022stability,
  title={Stability of multi-dimensional switched systems with an application to open multi-agent systems},
  author={Xue, Mengqi and Tang, Yang and Ren, Wei and Qian, Feng},
  journal={Automatica},
  volume={146},
  pages={110644},
  year={2022},
  publisher={Elsevier}
}

@article{de2022fundamental,
  title={Fundamental performance limitations for average consensus in open multi-agent systems},
  author={de Galland, Charles Monnoyer and Hendrickx, Julien M},
  journal={IEEE Transactions on Automatic Control},
  volume={68},
  number={2},
  pages={646--659},
  year={2022},
  publisher={IEEE}
}

@inproceedings{deplano2023unified,
  title={A unified approach to solve the dynamic consensus on the average, maximum, and median values with linear convergence},
  author={Deplano, Diego and Bastianello, Nicola and Franceschelli, Mauro and Johansson, Karl H},
  booktitle={2023 62nd IEEE Conference on Decision and Control (CDC)},
  pages={6442--6448},
  year={2023},
  organization={IEEE}
}

@article{bastianello2024robust,
  title={Robust Online Learning over Networks},
  author={Bastianello, Nicola and Deplano, Diego and Franceschelli, Mauro and Johansson, Karl H},
  journal={IEEE Transactions on Automatic Control},
  year={2024},
  publisher={IEEE}
}

@inproceedings{rai2023distributed,
  title={Distributed algorithms for edge-agreements: More than consensus},
  author={Rai, Ayush and Mou, Shaoshuai},
  booktitle={2023 62nd IEEE Conference on Decision and Control (CDC)},
  pages={4417--4422},
  year={2023},
  organization={IEEE}
}

@article{raja2021payoff,
  title={Payoff distribution in robust coalitional games on time-varying networks},
  author={Raja, Aitazaz Ali and Grammatico, Sergio},
  journal={IEEE Transactions on Control of Network Systems},
  volume={9},
  number={1},
  pages={511--520},
  year={2021},
  publisher={IEEE}
}

@article{elsner1992convergence,
  title={Convergence of sequential and asynchronous nonlinear paracontractions},
  author={Elsner, Ludwig and Koltracht, Israel and Neumann, Michael},
  journal={Numerische mathematik},
  volume={62},
  number={1},
  pages={305--319},
  year={1992},
  publisher={Springer}
}

@inproceedings{abdelrahim2017max,
  title={Max-consensus in open multi-agent systems with gossip interactions},
  author={Abdelrahim, Mahmoud and Hendrickx, Julien M and Heemels, WPMH},
  booktitle={2017 IEEE 56th Annual Conference on Decision and Control (CDC)},
  pages={4753--4758},
  year={2017},
}

@article{Parikh14,
	title = {Proximal Algorithms},
	volume = {1},
	number = {3},
	journal = {Foundations and Trends® in Optimization},
	author = {Parikh, Neal and Boyd, Stephen},
	year = {2014},
	pages = {127--239}
}

\section*{Appendix:\\ Proofs of Lemmas~\ref{lem:DOOT_TSI}-\ref{lem:DOOT_BAR} and Theorem~\ref{th:mainth_omas_opt}}
\setcounter{subsection}{0}
The standard iteration of \alg is that of DOT-ADMM proposed by the same authors in~\cite{bastianello2024robust}, which is given by (cfr.~\cite[Eq. (11)]{bastianello2024robust})
\begingroup
\medmuskip=0mu
\thinmuskip=0mu
\thickmuskip=0mu
\begin{equation*}\label{eq:compactform}
\begin{aligned}
x_{k} &= \Fs_{k}(x_{k-1}) = \left[ (1 - \alpha) I - \alpha P_{k-1} \right]  x_{k-1} + 2 \alpha \rho P_{k-1} A_{k-1}  y_{k-1}\\
y_{k} &= \Ps_k(x_k) = \prox_{f_k}^{1/\rho \eta_k}(D_{k} A_{k}^\top  x_{k})
\end{aligned}
\end{equation*}
\endgroup
where:
\begin{itemize}
    \item the operator $\prox_{f_k}^{1/\rho \eta_k} : \rea^{\J_k} \to \rea^{\J_k}$ with ${\J_k=\{(i,\ell):i\in\V_k\text{ and } \ell=1,\ldots,p\}}$ applies block-wise the proximal of the local costs $f^i_k$;
    \item the matrix ${A_k = \Lambda \otimes I_p\in\{0,1\}^{\I_k \times \J_k}}$ is given\footnote{The symbol $\otimes$ denotes the Kronecker product.} by ${\Lambda\in\{0,1\}^{\E_k\times \V_k}}$ defined block-wise for $i\in\V_k$ and $j\in\N^i_k$ by the matrices $\Lambda^{ij}=[e_i,e_j]^\top$ for $j>i$ where $e_\ell\in\{0,1\}^{\V_k}$ denotes a canonical vector;
    \item the matrix ${D_k\in\rea^{\J_k\times \J_k}}$ is given by ${D_k = \blkdiag\{ (\rho\eta^i_k)^{-1} I_p \}_{i = 1}^n}$;
    \item the matrix ${P_k\in \{0,1\}^{\I_k \times \I_k} }$ is a squared block-diagonal matrix given by $P_k=I_{\xi_k/2}\otimes (\Pi \otimes I_p)$ with $\Pi=[0\ 1; 1\ 0]$.
\end{itemize}

\subsection{Proof of Lemma~\ref{lem:DOOT_TSI}}
  Let $\hat{x}_{k}$ be a fixed point of standard iteration, namely
    $$
    \hat{x}_{k}=\Fs_{k+1}(\hat{x}_{k}).
    $$
    Since the updates of DOT-ADMM are the result of the application of the Peaceman-Rachford operator to the dual of the distributed version of the problem in~\eqref{eq:optprob}, then the output variable at a fixed point is equal to $(\bone_{n_k}\otimes y^\star_{k})$ where $y^\star_{k}\in \Y^\star_k$ is a solution to problem in~\eqref{eq:optprob} (cfr.~\cite[Therorem 26.11]{Ryu16}), namely,
    $$
    (\bone_{n_k}\otimes y^\star_{k}) = \prox_{f_k}^{1/\rho \eta_k}(D_k A_k^\top  \hat{x}_{k}).
    $$
    Consequently, the fixed point $\hat{x}_{k}$ of the standard iteration satisfies
    $$
    (I+{P_{k}})\hat{x}_{k} = 2\rho P_{k}A_{k} (\bone_{n_k}\otimes y^\star_{k}),
    $$
    thus completing the first part of the proof. 
    Now, we need to prove that the following is bounded from above,
    $$
    \begin{aligned}
        \dsh(\hat{\X}_{k},\hat{\X}_{k-1}) & = \sup_{z\in\rea^{\I_k}} d(\proj(z,\hat{\X}_{k}),\proj(z,\hat{\X}_{k-1})).
    \end{aligned}
    $$
    Let us denote by $\hat{z}_k$ the projection of $z\in\rea^{\I_k}$ onto $\hat{\X}_k\subseteq \rea^{\I_k}$, namely $\hat{z}_k = \proj(z,\hat{\X}_k)$, then by~\cite[Section 6.2.2]{Parikh14} it holds
    $$
    \hat{z}_k = z - (I+P_k)^\dag((I+P_k)z-2\rho P_kA_k(\bone_{n_k}\otimes y^\star_{k})).
    $$
    Since $I+P_k=I_{\xi_k/2}\otimes (J \otimes I_p)$ where $J=(\Pi+I_2)$ is a matrix with all ones with pseudoinverse $J^\dag=\frac{1}{4}J$, we can decompose $\hat{z}_k$ into vectors $\hat{z}^{ij}_k\in\rea^{2p}$ where $i\in \V_k$, $j\in\N^i_k$ such that $j>i$, given by
    $$
    \hat{z}^{ij}_k = z^{ij} - \frac{1}{2}(J \otimes I_p)(z^{ij} - \rho (\Lambda^{ij}\otimes I_p) (\bone_{n_k}\otimes y^\star_{k})).
    $$
    Similarly, the components of $\proj(z,\hat{\X}_{k-1})$ belonging to $\I_k$ are given by
    $$
    \hat{z}^{ij}_{k-1} = z^{ij} - \frac{1}{2}(J \otimes I_p)(z^{ij} - \rho (\Lambda^{ij}\otimes I_p) (\bone_{n_k}\otimes y^\star_{k-1}).
    $$
    Thus, $d(\proj(z,\hat{\X}_{k}),\proj(z,\hat{\X}_{k-1}))=$
    \begingroup
    \medmuskip=0mu
    \thinmuskip=0mu
    \thickmuskip=0mu
    \allowdisplaybreaks
    \begin{align*}
        & = \sqrt{\sum_{i\in\Vr_k}\sum_{i<j\in \R^i_k}\norm{\hat{z}^{ij}_k-\hat{z}^{ij}_{k-1}}^2 }\\
        &= \sqrt{\sum_{i\in\Vr_k}\sum_{i<j\in \R^i_k} \frac{\rho^2}{4}\norm{(J \otimes I_p) (\Lambda^{ij}\otimes I_p) (\bone_{n_k}\otimes(y^\star_{k}-y^\star_{k-1}))}^2 }\\
        &\overset{(i)}{\leq} \sqrt{\sum_{i\in\Vr_k}\sum_{i<j\in \R^i_k} \frac{\rho^2}{4}\norm{(J \otimes I_p) }^2\norm{(\Lambda^{ij}\otimes I_p) (\bone_{n_k}\otimes(y^\star_{k}-y^\star_{k-1}))}^2 }\\
        &= \sqrt{\sum_{i\in\Vr_k}\sum_{i<j\in \R^i_k} \frac{\rho^2}{4}\norm{(J \otimes I_p) }^2\norm{\begin{bmatrix}
            y^\star_{k}-y^\star_{k-1} \\ y^\star_{k}-y^\star_{k-1}
        \end{bmatrix}}^2 }\\
        & \overset{(ii)}{\leq}  \sqrt{\sum_{i\in\Vr_k}\sum_{i<j\in \R^i_k} \frac{\rho^2\sigma^2}{2}\norm{J \otimes I_p}^2  }\overset{(iii)}{=}  \sqrt{\sum_{i\in\Vr_k}\sum_{i<j\in \R^i_k} \frac{\rho^2\sigma^2}{2}\norm{J}^2\norm{I_p}^2  } \\
        & \overset{(iv)}{=}  \sqrt{\sum_{i\in\Vr_k}\sum_{i<j\in \R^i_k} \frac{\rho^2\sigma^2}{2}\norm{J}^2} \overset{(v)}{\leq}  \sqrt{\sum_{i\in\Vr_k}\sum_{i<j\in \R^i_k} \frac{\rho^2\sigma^2}{2}\norm{J}_1\norminf{J} } \\
        & \overset{(vi)}{=} \sqrt{\sum_{i\in\Vr_k}\sum_{i<j\in \R^i_k} \frac{\rho^2\sigma^2}{2}\cdot 2 \cdot 2 }  = \rho \sigma\sqrt{\sum_{i\in\Vr_k}\sum_{i<j\in \R^i_k} 2} \overset{(vii)}{\leq} \rho\sigma \sqrt{\abs{\R_k}} 
    \end{align*}
    \endgroup
    where $(i)$ follows by sub-multiplicativity of the norm; $(ii)$ holds by Assumption~\ref{as:all}(iii); $(iii)$ follows by~\cite[Theorem 8]{lancaster1972norms}; $(iv)$ holds because the $2$-norm of an identity matrix is equal to $1$; $(v)$ follows by the Riesz-Thorin Theorem \cite[ Theorem 4.3.1]{Serre12}; $(vi)$ follows by the fact that $\norm{M}_1$ and $\norm{M}_\infty$ are, respectively, the row- and column- sum of the absolute values of the matrix $M$, and from the fact that $J$ has exactly $2$ ones in each row and each column while $\Lambda^{ij}$ has exactly $1$ one in each row and each column; $(vii)$ holds because the number of remaining channels at time $k$ is equal to $\abs{\R_k}$. Therefore, the TSI is bounded with $B=\rho\sigma/\sqrt{p}$, completing the proof.

\subsection{Proof of Lemma~\ref{lem:DOOT_BAR}}
   \alg requires that state components $x^{ij}_k$ of all arriving agents $i\in\Va_k$ and new state components of the remaining agents $i\in\Vr_k$, $j\in \A^i_k$ are initialized to the minimizer $y^{i,\star}_k$ of the local cost $f^i_k$, scaled by the penalty parameter $\rho>0$, namely,
    $$
    x^{ij}_k = \rho y^{i,\star}_k,\text{ where }  y^{i,\star}_k \in \Y^{i,\star}_k = \left\{y:f^i_k(y) =  \min_{z\in\rea^{\P}} f^i_k(z)\right\}.
    $$
    Let $x_k^{\textsc{A}}\in \rea^{\A_k}$ be the vector stacking the components of all arriving labels, and denote $\hat{x}^{\textsc{A}}_k = \proj(x^{\textsc{A}}_k,\hat{\X}_{k})$. Further, let us decompose $x^{\textsc{A}}_k,\hat{x}^{\textsc{A}}_k$ into vectors $x^{,ij}_k,\hat{x}^{ij}_k\in\rea^{\P}$ where $i\in \V_k$, $j\in\A^i_k$.
    Then, by Lemma~\ref{lem:DOOT_TSI} it holds that:
    \begingroup
    \medmuskip=0mu
    \thinmuskip=0mu
    \thickmuskip=0mu
    \allowdisplaybreaks
    \begin{align*}
        &d(x^{\textsc{A}}_k,\hat{\X}_{k})  = d(x^{\textsc{A}}_k,\proj(x^{\textsc{A}}_k,\hat{\X}_{k})) = d(x^{\textsc{A}}_k,\hat{x}^{\textsc{A}}_k)\\
        & = \sqrt{\sum_{i\in\V_k} \sum_{i<j\in\A^i_k} \norm{\begin{bmatrix}
            x^{ij}_k\\x^{ji}_k
        \end{bmatrix}-\begin{bmatrix}
            \hat{x}^{ij}_k\\\hat{x}^{ji}_k
        \end{bmatrix}}^2}\\
        &  \overset{(1)}{\leq} \sqrt{\sum_{i\in\V_k} \sum_{i<j\in\A^i_k} \norm{\frac{1}{2}(J \otimes I_p)(\begin{bmatrix}
            x^{ij}_k\\x^{ji}_k
        \end{bmatrix} - \rho (\Lambda^{ij}\otimes I_p) (\bone_{n_k}\otimes y^\star_{k}))}^2} \\ 
        & = \sqrt{\sum_{i\in\V_k} \sum_{i<j\in\A^i_k} \norm{\frac{\rho}{2}(J \otimes I_p)\left(\begin{bmatrix}
        y^{i,\star}_k \\ y^{j,\star}_k
        \end{bmatrix} - (\Lambda^{ij}\otimes I_p) (\bone_{n_k}\otimes y^\star_{k})\right)}^2}\\
        & = \sqrt{\frac{\rho^2}{4}\sum_{i\in\V_k} \sum_{i<j\in\A^i_k}  \norm{(J \otimes I_p)\left(\begin{bmatrix}
        y^{i,\star}_k \\ y^{j,\star}_k
        \end{bmatrix} - \begin{bmatrix}
            y^\star_{k} \\ y^\star_{k}
        \end{bmatrix}\right)}^2}\\
        & \overset{(2)}{\leq} \sqrt{\frac{\rho^2}{4}\norm{(J \otimes I_p)}^2\sum_{i\in\V_k} \sum_{i<j\in\A^i_k}   \norm{\begin{bmatrix}
        y^{i,\star}_k - y^\star_{k}\\ y^{j,\star}_k - y^\star_{k}
        \end{bmatrix}}^2}\\
        & \overset{(3)}{\leq} \sqrt{\rho^2 \sum_{i\in\V_k} \sum_{i<j\in\A^i_k}  \norm{\begin{bmatrix}
        y^{i,\star}_k - y^\star_{k}\\ y^{j,\star}_k - y^\star_{k}
        \end{bmatrix}}^2} \overset{(4)}{\leq} \sqrt{\rho^2 \sum_{i\in\V_k} \sum_{i<j\in\A^i_k}  2\omega^2}\\
        & = \rho \omega \sqrt{\sum_{i\in\V_k} \sum_{i<j\in\A^i_k} 2 } \overset{(6)}{\leq} \rho\omega \sqrt{\abs{\A_k}},
    \end{align*}
    \endgroup
    where $(1)$ hold by Lemma~\ref{lem:DOOT_TSI}; $(2)$ holds by triangle inequality; $(3)$ holds as explained in steps $(3)-(6)$ at the end of the proof of Lemma~\ref{lem:DOOT_TSI}; $(4)$ holds by Assumption~\ref{as:all}(iv); 
    $(5)$ holds because the number of arriving channels at time $k$ is equal to $\abs{\A_k}$. Therefore, the arrival process is bounded with $H=\frac{\rho\omega}{\sqrt{p}}$, completing the proof.

\subsection{Proof of Theorem~\ref{th:mainth_omas_opt}}
We compute an upper bound to the distance of the estimation vector $y_k = [y^i_k,\cdots,y^{n_k}_k]\in\rea^{n_k}$ of the whole network from the consensus state on the solutions, which we call ${\mathcal{C}^\star_k := \{\bone_{n_k}\otimes y^\star_{k} \ \mid \ y^\star_{k}\in\Y^\star_k\}}$.
By~\cite[Proposition~3]{Bastianello21}, for each point $ \hat{x}_k \in \hat{\X}_k$ in the TSI there is a solution $y^\star_k\in\Y^\star_k$ to the problem in~\eqref{eq:optprob} such that $(\bone_{n_k}\otimes y^\star_{k})=\Ps_{k}(\hat{x}_k)$ for all $i\in\V_k$.
Thus we can write
\begingroup
\medmuskip=0mu
\thinmuskip=0mu
\thickmuskip=0mu
\begin{align*}
    d(y_k,\C^\star_k) &=  \inf_{ z \in \C^\star_k} \norm{ y_k -  z}  \overset{(i)}{\leq} \norm{ y_k -  (\bone_{n_k}\otimes y^\star_{k})} = \norm{\Ps_{k}( x_{k}) - \Ps_{k}(\hat{x}_{k})} \\
    &\overset{(ii)}{\leq} \norm{\prox_{f_k}^{1/\rho \eta_k}(D_k A_{k}^\top  x_{k}) - \prox_{f_k}^{1/\rho \eta_k}(D_kA_{k}^\top  \hat{x}_{k})} \\
    &\overset{(iii)}{\leq} \norm{D_kA_{k}^\top ( x_{k} -  \hat{x}_{k} ) }
    \overset{(iv)}{\leq} \frac{1}{\rho} \norm{\rho D_kA_{k}^\top} d(x_{k},\hat{\X}_k) \\
    & \overset{(v)}{=} \frac{1}{\rho\sqrt{\min_{i\in\V_k}\eta^i_k}} d(x_{k},\hat{\X}_k) \leq \frac{1}{\rho} d(x_{k},\hat{\X}_k)
\end{align*}
\endgroup
where $(i)$ holds since $ y_k^\star\in\Y_k^\star$ and because of the consensus constraint; $(ii)$ holds with by definition of $\Ps_k$; $(iii)$ follows by the non-expansiveness of the proximal, and $(iv)$ holds by choosing
$ \hat{x}_k =\arginf_{ y \in \hat{\X}_k} \norm{x_{k} -  y}$;
$(v)$ holds because matrix $\rho D_kA_{k}^\top$ is row-stochastic and its columns sums up to $1/\eta^i_k$.

This means that the linear convergence of $ y_k$ to a neighborhood of $\C^\star_k$ is implied by that of $x_k$ to a neighborhood of $\hat{\X}_k$, which follows from Theorem~\ref{th:mainth_omas} and Lemmas~\ref{lem:DOOT_TSI}-\ref{lem:DOOT_BAR}.
Moreover, the convergence radius becomes time-varying and depending on the number of agents because $y^i_k$ and $x_k$ have different dimensions, indeed, 
\begingroup
$$
\begin{aligned}
        \limsup_{k\rightarrow \infty} \ \frac{d(y_k,\C^\star_k)}{\sqrt{pn_k}} & \leq \limsup_{k\rightarrow \infty} \ \frac{d(x_{k},\hat{\X}_k)}{\rho\sqrt{pn_k}} \leq \frac{\sqrt{p\xi_k}}{\rho\sqrt{pn_k}}\frac{B+H}{1-\frac{\gamma}{\beta}}\\
        &\leq \frac{n_k}{\rho\sqrt{n_k}}\frac{B+H}{1-\frac{\gamma}{\beta}} = \frac{\sqrt{n_k}}{\rho}\frac{B+H}{1-\frac{\gamma}{\beta}}
\end{aligned}
$$
\endgroup
and, using $B=\rho\sigma$ in~\eqref{eq:TSI_bound}, $H=\rho\omega$ in~\eqref{eq:TSI_boundarr}, yields the thesis.

\end{document}